\documentclass[11pt]{article}
\usepackage{latexsym,amsfonts,amssymb,amsmath,amsthm}
\usepackage{graphicx}

\usepackage[usenames,dvipsnames]{color}
\usepackage{ulem,comment}

\usepackage{color}
\title{Propagation dynamics of a parabolic-elliptic chemotaxis model in a time-periodic shifting environment}

\author{Shuwen Xue\thanks{Shuwen Xue (corresponding author, E-mail: sxue@niu.edu)}\\
School of Mathematical and Statistical Sciences\\
Northern Illinois University\\
DeKalb, IL 60115, USA\\
and\\ 
Xiao-Qiang Zhao\thanks{Xiao-Qiang Zhao (E-mail: zhao@mun.ca)'s research is partially supported by the NSERC of Canada (RGPIN-2019-05648 and RGPIN-2025-04963). }\\
Department of Mathematics and Statistics\\
Memorial University of Newfoundland\\
St. John's, NL A1C 5S7, Canada.\\
}

\date{}

\begin{document}

\maketitle

\newtheorem{tm}{Theorem}[section]
\newtheorem{prop}{Proposition}[section]
\newtheorem{defin}{Definition}[section] % definition numbers are dependent on theorem numbers
\newtheorem{coro}{Corollary}[section]
\newtheorem{lem}{Lemma}[section]
\newtheorem{assumption}{Assumption}[section]
\newtheorem{rk}{Remark}[section]
\newtheorem{nota}[tm]{Notation}
\numberwithin{equation}{section}

\newcommand{\stk}[2]{\stackrel{#1}{#2}}
\newcommand{\dwn}[1]{{\scriptstyle #1}\downarrow}
\newcommand{\upa}[1]{{\scriptstyle #1}\uparrow}
\newcommand{\nea}[1]{{\scriptstyle #1}\nearrow}
\newcommand{\sea}[1]{\searrow {\scriptstyle #1}}
\newcommand{\csti}[3]{(#1+1) (#2)^{1/ (#1+1)} (#1)^{- #1
 / (#1+1)} (#3)^{ #1 / (#1 +1)}}
\newcommand{\RR}[1]{\mathbb{#1}}

\newcommand{\rd}{{\mathbb R^d}}
\newcommand{\ep}{\varepsilon}
\newcommand{\rr}{{\mathbb R}}
\newcommand{\alert}[1]{\fbox{#1}}
\newcommand{\eqd}{\sim}
\def\p{\partial}
\def\R{{\mathbb R}}
\def\N{{\mathbb N}}
\def\Q{{\mathbb Q}}
\def\C{{\mathbb C}}
\def\l{{\langle}}
\def\r{\rangle}
\def\t{\tau}
\def\k{\kappa}
\def\a{\alpha}
\def\la{\lambda}
\def\De{\Delta}
\def\de{\delta}
\def\ga{\gamma}
\def\Ga{\Gamma}
\def\ep{\varepsilon}
\def\eps{\varepsilon}
\def\si{\sigma}
\def\Re {{\rm Re}\,}
\def\Im {{\rm Im}\,}
\def\E{{\mathbb E}}
\def\P{{\mathbb P}}
\def\Z{{\mathbb Z}}
\def\D{{\mathbb D}}
\newcommand{\ceil}[1]{\lceil{#1}\rceil}
\def\mymathhyphen{{\hbox{-}}}

\begin{abstract}
The current paper is devoted to the study of the propagation dynamics of a parabolic-elliptic chemotaxis model in a shifting environment. We investigate the scenario in which the favorable and unfavorable habitats are spatially separated, with the favorable habitat shifting at a time-periodic speed. We first establish the vanishing, persistence, and spreading properties of solutions with compactly supported and front-like initial data. In particular, we show that if the self-limitation rate of the species is large relative to the chemotactic sensitivity and the species is initially confined to a bounded region, then the species eventually becomes extinct when the habitat shifts rightward with a large average speed. In contrast, when the average shifting speed of the habitat is moderate, the species can persist and spread along the shifting habitat into an increasingly large spatial region. Finally, we establish the existence and nonexistence of time-periodic forced wave solutions whose wave speed coincides with the average shifting speed of the habitat.
\end{abstract}


%\medskip
%\noindent {\bf 2020 Mathematics Subject Classification.} 35B35, 35B40, 35K57, 35Q92, 92C17. 	{\color{red} revise?}


\section{Introduction}

This work is concerned with the propagation dynamics of the following parabolic-elliptic chemotaxis model in a in a shifting environment with time-periodic shifting speed:
\begin{equation}\label{Keller-Segel-eq0}
\begin{cases}
u_t=u_{xx}-\chi( u v_{x})_{x}+u(r(t, x-s(t))-bu),\quad x\in\R\cr
0 =v_{xx}-  \lambda v +\mu u,\quad x\in\R,
\end{cases}
\end{equation}
where $ b, \lambda,\mu$, $\chi$ are positive constants, $s(t)$ is a
continuously differentiable  shifting function whose derivative $s'(t)$ is continuous and $\omega\hbox{-}{\rm periodic }$ in $t$. The functions $u(t,x)$ and $v(t,x)$ represent the densities of a biological species and a chemical substance, respectively, at time $t$ and location $x$.  The positive constant $\chi$ measures the chemotactic sensitivity of the species to the chemical substance produced by the species itself; the reaction term $u(r(t, x-s(t))-bu)$ in the first  equation of system \eqref{Keller-Segel-eq0} describes the local population dynamics of the species;
 %which depends on the density $u$, time $t$ and the shifting habitat with a time-periodic shifting speed; 
 $\lambda$  represents the degradation rate of the  chemical substance;  and $\mu$ is the rate at which the species produces the chemical substance.

System \eqref{Keller-Segel-eq0} describes the evolution of a biological species subject to chemotactic effects in a shifting environment whose shifting speed varies periodically in time. The directed movement of biological species or cells in response to gradients of a chemical substance is referred to as chemotaxis. Chemotaxis is observed in a wide variety of biological phenomena, ranging from microscopic cellular processes to macroscopic population-level behaviors, including cancer growth, embryonic development, immune responses, wound healing, neuronal migration, the search for food sources, predator avoidance, mate attraction, and population dynamics, among others (see \cite{Ei}). Because of its important role in these processes, chemotaxis has attracted considerable attention from biologists, ecologists, and mathematicians. Mathematical modeling of chemotaxis dates back to the pioneering work of Keller and Segel in the early 1970s on the aggregation of the slime mold \textit{Dictyostelium discoideum} (see \cite{KeSe1, KeSe2}). Since then, a wide variety of mathematical models have been proposed to describe chemotactic processes. Such models are commonly referred to as chemotaxis models or Keller-Segel models, and their mathematical analysis have been extensively studied. We refer the reader to \cite{BBTW, CaLiWa, HiPa, Hor, IssaShen20171, IssaShen20172, KurtShenXue, KJPainter} and the references therein for the mathematical study of chemotaxis models and their applications, and to the survey \cite{ArTy} for more recent developments.

Different properties of the function $r(\cdot, \cdot)$  correspond to different environmental scenarios. Throughout the whole paper, we always assume that

\medskip

\noindent {\bf (H1)} {\it $r(t, x)$ is globally H\"older continuous on $\R\times\R$ and $\omega$-periodic in $t$  for some $\omega>0$. Moreover, the limits $r(t, \pm\infty):=\lim_{x\to \pm\infty}r(t, x)$ exist uniformly for $t\in [0, \omega]$ and are finite. For each $t\in [0,\omega]$, $r(t, x)$ is nondecreasing in $x$ and $r(t, -\infty)<0<r(t, \infty)$.}

\medskip

Assumption {\bf (H1)} implies that, at any given time $t$, the living environment is divided into exactly two regions: a favorable region suitable for species growth (i.e.,  $\{x\in\R: r(t,x)>0\}$) and an unfavorable region unsuitable for species growth (i.e., $\{x\in\R: r(t,x)<0\}$). In particular, the environment is favorable near $\infty$ and unfavorable near $-\infty$. Moreover, the environment shifts at a time-periodic speed. Such problems may arise from the biological question of whether a species can survive environmental shifts caused by climate change associated with global warming or industrialization (see \cite{BDNZ, LBSF, APMAL}). They may also arise from the epidemiological question of whether the spread of a pathogen can keep pace with the movement of its host (see \cite{FaLoWu}).

Seasonal variation may cause the shifting speed of a favorable habitat to fluctuate periodically, thereby forcing a species to shift or expand its range in order to persist. For example, climate-induced periodic northward migration of the silk moth has been observed \cite{KhImYa}. Seasonal variation may also significantly affect the growth rate of a species. For instance, silk moth eggs hatch at a very low rate during winter because of low temperatures, whereas their hatching rates are much higher during spring and summer. In addition, species movement may be influenced by chemical substances produced by the species themselves. During the mating season, for example, female moths release pheromones whose scent attracts males and induces them to move toward regions of higher pheromone concentration. Motivated by these biological mechanisms, we consider a chemotaxis model in a time-periodically shifting environment.

It is of both biological and mathematical interest to understand how environmental heterogeneity and temporal variation affect the propagation dynamics of a species, including whether the species can persist, whether it can spread into increasingly large spatial regions and, if so, how fast it spreads, and whether it can propagate as a wave.

These questions have been extensively investigated for system  \eqref{Keller-Segel-eq0} in the absence of chemotaxis (i.e.,  $\chi=0$) and with a space-time homogeneous growth rate (i.e., $r(t, x-s(t))\equiv r>0$), in which case the system reduces to,
\begin{equation}\label{F-KPP}
u_{t}=u_{xx}+u(r-bu), \quad x\in\R.
\end{equation}
Equation \eqref{F-KPP} is referred to as  the {\it  Fisher-KPP  equation}, owing to the pioneering works of Fisher \cite{Fisher} and Kolmogorov, Petrowsky, and Piskunov \cite{KPP} on its traveling wave solutions and spreading properties.
The following results concerning the traveling wave solutions and spreading speed of \eqref{F-KPP} are well known. Equation \eqref{F-KPP} admits traveling wave solutions $u(t,x)=\phi(x-ct)$
connecting $\frac{r}{b}$ and $0$ (i.e., $\phi(-\infty)=\frac{r}{b}$ and $\phi(\infty)=0$) for every $c\geq 2\sqrt r$ and has no such traveling wave solutions of slower speeds (see \cite{Fisher, KPP}). Moreover, for any nontrivial nonnegative solution $u(t,x)$ of \eqref{F-KPP} whose initial function $u(0, x)$ has compact support,  one has
$$
\lim_{t\to \infty}\sup_{|x| \ge ct}u(t,x)=0 \quad \forall \,\, c>2\sqrt r
$$
and
$$
\lim_{t\to \infty}\sup_{ |x| \le ct}|u(t,x)-\frac{r}{b}|=0\quad \forall\,\, 0<c<2\sqrt r
$$
(see \cite{ArWe1, ArWe2}). In literature, $2\sqrt r$ is called the {\it asymptotic speed of spread}, or simply the {\it spreading speed}, of \eqref{F-KPP}, which was first introduced by Aronson and Weinberger \cite{ArWe1}.  Since their pioneering works, a huge amount of research has been carried out toward the front propagation dynamics of different types of Fisher–KPP equations,  see, e.g.,  \cite{LiYiZh, LiZh, LiZh1, Nad, NoRuXi} and references therein for time and/or space periodic Fisher–KPP type equations, and  \cite{BeHa, HuSh, NadRoss, Shen1, Shen2} and references therein for general time and/or space dependent Fisher-KPP type equations.

Observe that in the absence of chemotaxis (i.e.,  $\chi=0$), the first equation in \eqref{Keller-Segel-eq0} is decoupled from the second one, and system \eqref{Keller-Segel-eq0} reduces to a Fisher-KPP equation in a shifting environment. In recent years, considerable attention has been devoted to understanding the effects of climate change on species persistence and spreading, and various patterns of environmental change have been proposed. A typical scenario is that the favorable environment shifts at a constant speed. Mathematical models describing such environments were introduced and studied in \cite{BDNZ, APMAL}.  More precisely,  Berestycki et al. \cite{BDNZ} considered a Fisher-KPP type equation of the form
\begin{equation}\label{B-fisher-kpp1}
u_t= Du_{xx}+f(u, x-ct), \quad x\in\R.
\end{equation}
where $c\geq 0$ is the shifting speed of the environment, and $f$ is the growth rate function taking, for example, the form
$$
f(u, x)=\begin{cases}
au(1-\frac{u}{K}),\quad &0\leq x \leq L,\\
-ru,\quad &x<0 \,\ {\rm and }\,\ x>L
\end{cases}
$$
for some $L >0$ standing for the size of the climate envelope within which species can grow. They showed that solutions vanish as time elapses when the generalized principal eigenvalue of the operator $u\to Du_{xx} + cu_{x} + f_{u}(x,0)u$, is non-positive. When the generalized principal eigenvalue is positive, a unique travelling wave with speed $c$
exists and attracts solutions. They also studied the dependence of the generalized principal eigenvalue on the shifting speed $c$ and the size of the climate envelope $L$.
It should be pointed out that the paper \cite{APMAL} addresses the same question as in \cite{BDNZ}, but focuses on the effect of a moving climate on the outcome of competitive interactions between two species. 

When $f$ in \eqref{B-fisher-kpp1} takes the special form $f(u, x-ct)=u(r(x-ct)-u)$, Li et al. \cite{LBSF} investigated the spreading dynamics of solutions of \eqref{B-fisher-kpp1} in a shifting environment with $r(-\infty)<0<r(\infty)$ and $c\geq 0$. More precisely, they showed that if $c>2\sqrt{r(\infty)}$, then the species eventually becomes extinct, whereas if $0<c<2\sqrt{r(\infty)}$, then the species can persist and spread along the shifting habitat. Later, the existence of a nondecreasing traveling wave solution connecting $0$ and $r(\infty)$, whose wave speed coincides with the habitat shifting speed $c$, was established in \cite{HHXZ}. Such a wave is referred to as a {\it forced wave solution}. Moreover, It was shown that the forced wave solution is unique and globally exponential stable (see \cite[Theorem 2.3]{WaZh}). 
When $f(u, x)$ in \eqref{B-fisher-kpp1} is asymptotically of KPP type as $x\to-\infty$, Berestycki and Fang \cite{BeFa} established the existence and multiplicity of forced wave solutions, as well as their attractivity.

The works \cite{BDNZ, LBSF, APMAL} have been extended in several aspects (see \cite{AlBeRa, BhRl1, BhRl2, BoGi, BoNa, HuYiZo, JiangZhao, Vo, ZhangZhao} and reference therein). In a time-periodic shifting environment, Fang, Peng and Zhao \cite{FaPeZh} investigated the propagation dynamics of solutions to a nonautonomous reaction-diffusion equation and established the existence, uniqueness, and stability of periodic forced waves. In an almost-periodic shifting environment, the authors of \cite{ShShXuZh} studied the propagation dynamics of a reaction-diffusion equation. More precisely, they established a persistence criterion in terms of the sign of the approximate top Lyapunov exponent and proved the existence of a unique forced wave solution. They also investigated the effects of fluctuations in the shifting speed on species evolution and showed that such fluctuations have an adverse effect on species persistence. Analogous results have been obtained for integro-difference equations (see \cite{BoLe, HaZhHiKo, LeMaSh, LiWu, PhKo, ZhKo1, ZhKo2} and the references therein), nonlocal dispersal equations (see \cite{LeShZh, LiWaZh, QiLiWa, WaZh, ZhZh} and the references therein), and lattice differential equations (see \cite{HuLi, PaWu} and the references therein). Related studies have also been carried out for competition models in shifting environments (see \cite{DoFaLiLi, YuWaZo, ZhWaYa, WuWaZo, ZhWaDo} and the references therein) and free boundary problems in shifting environments (see \cite{DuWeZh, HuHaSoDu, LeDu, LeiNiDoDu} and the references therein). It is also worth mentioning that, in an abstract framework, Yi and Zhao \cite{YiZh2020, YiZh2026} established the existence of spatially inhomogeneous steady states and asymptotic propagation properties for a broad class of monotone evolution systems with/without spatial translation invariance. Their general theory can be applied to study the propagation dynamics of autonomous and periodic cooperative systems, with or without time delays, in shifting environments.

There are also quite a few works on  the spreading properties of system
\eqref{Keller-Segel-eq0} in the presence of chemotaxis (i.e.,  $\chi \neq 0$) and space-time homogeneous growth rate (i.e. $r(t, x-s(t))\equiv r>0$) (see \cite{SaSh1, SaSh2, SaShXu}). In particular, it was proved in \cite{SaShXu} that,  if $b>\chi\mu$ and $b\geq \big(1+\frac{1}{2}\frac{(\sqrt{r}-\sqrt{\lambda})_+}{(\sqrt{r}+\sqrt{\la})}\big)\chi\mu$ where $(\sqrt{r}-\sqrt{\lambda})_+$ denotes the positive part of $\sqrt{r}-\sqrt{\lambda}$,
then
$2\sqrt r$ is the spreading speed for solutions of system \eqref{Keller-Segel-eq0} with nonnegative continuous initial data $u_0$ of nonempty compact support in the sense that 
 $$
 \lim_{t\to\infty}\sup_{|x|\ge ct}u(t, x;u_0)=0\quad \forall\,\, c>2\sqrt{r}
 $$
 and
 $$
 \lim_{t\to\infty}\inf_{|x|\le ct} u(t, x;u_0)>0\quad \forall \,\, 0<c<2\sqrt{r},
 $$
 where $(u(t, x;u_0),v(t, x;u_0))$ is the unique global classical solution of system \eqref{Keller-Segel-eq0} with $u(0, x;u_0)=u_0(x)$.
 It was also proved that  if  $b>2\chi\mu$  and $\lambda \geq r$, then  $2\sqrt r$ is the minimal speed  of the  traveling wave solutions of system \eqref{Keller-Segel-eq0}  connecting $(0,0)$ and $(\frac{r}{b},\frac{\mu}{\lambda}\frac{r}{b})$. More precisely,
 for every $c\ge 2\sqrt r$, \eqref{Keller-Segel-eq0} has a  traveling wave solution  connecting these two states, whereas no such traveling wave solution exists with speed $c<2\sqrt r$. In particular, if $\lambda \geq r$ and $b>\chi\mu$, or $\lambda<r$ and  $b\geq \big(1+\frac{1}{2}\frac{(\sqrt{r}-\sqrt{\lambda})}{(\sqrt{r}+\sqrt{\la})}\big)\chi\mu$, then the chemotaxis neither speeds up nor slows down the spatial spreading of the corresponding Fisher-KPP equation \eqref{F-KPP}. The authors of \cite{SaSh3, SaSh4, SaSh5} studied the asymptotic dynamics of \eqref{Keller-Segel-eq0} with space-time dependent logistic source. In particular, in the case of space homogeneous and time-periodic logistic source (i.e. $r(t, x-s(t))\equiv r(t)$), it was proved in \cite{SaSh4} that,  if $b>2\chi\mu$, there exists a unique spatially homogeneous strictly positive entire solution that is uniformly and exponentially stable with respect to strictly positive perturbations.

% In the rest of this introduction, we introduce some standing notations and state the main results on the persistence and  spreading speeds  of \eqref{Keller-Segel-eq0}.

The study of chemotaxis models in shifting environments is relatively recent. The authors of \cite{ShXu1, ShXu2} investigated system \eqref{Keller-Segel-eq0} with $r(t, x-s(t))$ replaced by $r(x-ct)$. They assumed that the limits $r(\pm \infty):=\lim_{x\to \pm\infty}r(x)$ exist and that either $r(-\infty)<0<r(\infty)$ or $r(\pm \infty)<0$. Let $r^{*}=\sup_{x\in\R} r(x)$. In the case $r(-\infty)<0<r(\infty)$, they showed that if the shifting speed $c>c^{*}:=2\sqrt{r^{*}}$, then the species eventually becomes extinct. If $-c^{*}\leq c<c^{*}$, then the species persists and spreads along the shifting habitat. If $c<-c^{*}$, then the species will spread towards both directions. They also established the existence of a forced wave solution of system \eqref{Keller-Segel-eq0}. In the case $r(\pm\infty)<0$, let $\lambda_{\infty}$ denote the generalized principle eigenvalue of the operator $u\to u_{xx}+cu_{x}+r(x)u$. It was shown that if $\lambda_{\infty}<0$ and the degradation rate of the chemoattractant is large, then the species eventually becomes extinct. If $\lambda_{\infty}>0$, then the species persists in the vicinity of the favorable habitat. In the persistence case, it was further proved that, provided $\chi$ is sufficiently small, system \eqref{Keller-Segel-eq0} admits a forced wave solution connecting $(0,0)$ to $(0,0)$.
%When the species persists, it is proved that system \eqref{Keller-Segel-eq0} has a forced wave solution connecting $(0, 0)$ and $(0, 0)$ provided that $\chi$ is sufficiently small.

The objective of the present paper is to investigate the persistence and spreading speeds of system \eqref{Keller-Segel-eq0}, as well as the existence and nonexistence of forced wave solutions, in the presence of chemotaxis and a shifting environment with time-periodic shifting speed. 
To state our main results, we first introduce some notations and definitions. Let
$$
C^b_{\rm unif}(\R)=\{ u\in C(\R) : u\ \text{is uniformly continuous and bounded on $\R$}\}.
$$
For every $u\in C^b_{\rm unif}(\R)$,  define $\|u\|_{\infty}:=\sup_{x\in\R}|u(x)|$. For each nonnegative initial function $u_0\in C^b_{\rm unif}(\R)$, we denote by $(u(t,x;u_0),v(t,x;u_0))$ the classical solution of system \eqref{Keller-Segel-eq0} satisfying $u(0,x;u_0)=u_0(x)$ for every $x\in\R$. Comparison principle implies that $u(t,x;u_0)\geq 0$ and $v(t,x;u_0)\geq 0$ whenever the solution is defined.Throughout this paper, we are interested in nonnegative classical solutions of \eqref{Keller-Segel-eq0}  since both functions $u(t,x)$ and $v(t,x)$ represent density functions.

By arguments similar to those used in (\cite[Theorem 1.1 and Theorem 1.5]{SaSh1}), one can establish the following existence and uniqueness result for system \eqref{Keller-Segel-eq0} with nonnegative initial data.

\begin{prop}\label{existence-uniqueness}
For every nonnegative initial function $u_0\in C^b_{\rm unif}(\R)$, there is a unique maximal time $T_{max}$ such that $(u(t,x;u_0),v(t,x;u_0))$ is defined for every $x\in\R$ and $0\le t<T_{\max}$. Moreover, if $\chi\mu<b$, then $T_{max}=\infty$ and the solution is globally bounded.
\end{prop}

Throughout this paper, %we assume that $r(x)$ is as in {\bf Case 1} or {\bf Case 2}.  
we use the following well-defined constants:
\begin{equation}
\label{r-c-star-1}
c=\frac{1}{\omega}\int_{0}^{\omega}s'(t)dt,\quad
r^*=\frac{1}{\omega}\int_{0}^{\omega}r(t,\infty)dt,\quad c^*=2\sqrt {r^*}
\end{equation}
and 
\begin{equation}
\label{r-c-star-2}
 \bar r=\max_{t\in [0,\omega]}r(t,\infty), \quad \underline r=\min_{t\in [0,\omega]}r(t,\infty).
\end{equation}
In the absence of chemotaxis (i.e., $\chi=0$), the limiting equation of system \eqref{Keller-Segel-eq0} as $x\to\infty$ is the following time-periodic Fisher-KPP equation:
\begin{equation}
\label{KPP-periodic}
u_{t}=u_{xx}+u(r(t,\infty)-bu).
\end{equation}
The corresponding spatially homogeneous equation is
\begin{equation}
\label{ODE-periodic}
u_{t}=u(r(t,\infty)-bu).
\end{equation}
It follows from \cite[Theorem 5.2.1]{Zhao} that \eqref{ODE-periodic} admits a unique positive $\omega$-periodic solution $p(t)$, which is globally stable with respect to positive initial values. In view of \cite[Theorems 4.1 and 4.2]{LiYiZh}, the time-periodic Fisher-KPP equation \eqref{KPP-periodic} has a spreading speed $c^{*}$, which is also the minimal wave speed of 
time-periodic positive traveling wave solutions connecting $p(t)$ and $0$.

We also impose the following assumption:

\medskip

\noindent {\bf (H2)} $b>\chi\mu$ and $b\ge \big(1+\frac{1}{2}\frac{(\sqrt{r^*}-\sqrt{\lambda})_+}{(\sqrt{r^*}+\sqrt{\lambda})}\big)\chi\mu$.

\medskip

It is readily seen that $b\ge \big(1+\frac{1}{2}\frac{(\sqrt{r^*}-\sqrt{\lambda})_+}{(\sqrt{r^*}+\sqrt{\lambda})}\big)\chi\mu$ implies $b\ge \chi\mu$. Moreover, if  $b>\chi\mu$ and $\lambda \geq r^{*}$, then $(\sqrt{r^*}-\sqrt{\lambda})_{+}=0$, and hence {\bf (H2)} is automatically satisfied.
%$b\ge \big(1+\frac{1}{2}\frac{(\sqrt{r^*}-\sqrt{\lambda})_+}{(\sqrt{r^*}+\sqrt{\lambda})}\big)\chi\mu $ (and hence {\bf (H2)}). 
Biologically, the condition $\lambda \geq r^{*}$ means that the degradation rate of the chemoattractant is at least as large as the temporal average of the limiting intrinsic growth rate of the species over one period. The conditions in {\bf (H2)} indicate that the chemotactic effect is small relative to the self-limitation rate of the biological species.

\smallskip

%Our first result concerns the vanishing, persistence and spreading of solutions of \eqref{Keller-Segel-eq0} for nonempty compact supported initial data.
Our first result concerns the vanishing, persistence, and spreading dynamics of solutions to \eqref{Keller-Segel-eq0} with nonempty compactly supported initial data.

\begin{tm}\label{compact-support}
Suppose that {\bf (H1)} and {\bf (H2)} hold, and let \(u_0\in C^b_{\rm unif}(\R)\) be nonnegative with nonempty compact support. Then the following statements hold:
\begin{itemize}
\item[(1)]
If  $c>c^*$, then $\lim_{t\to\infty}u(t,x;u_0)=0$ uniformly for $x\in\R$.

\item[(2)]
If ${-c^*\leq}c<c^*$, then for every $0<\varepsilon<\frac{c^*-c}{2}$, there hold
$$
\lim_{t\to\infty}\sup_{x\le (c-\varepsilon)t}u(t,x;u_0)=\lim_{t\to\infty}\sup_{x\ge (c^*+\varepsilon)t}u(t,x;u_0)=0,
$$
and
$$
\liminf_{t\to\infty}\inf_{(c+\varepsilon)t\le x\le (c^*-\varepsilon)t}u(t,x;u_0)>0.
$$
 Moreover, if $b>2\chi\mu$, then
$$\lim_{t\to\infty} \sup_{(c+\varepsilon)t\le x\le (c^*-\varepsilon)t}|u(t,x;u_0)-p(t)|=0.$$
\item[(3)]
 If $c<-c^*$, then for every $0<\varepsilon<c^*$, there hold
 $$
\lim_{t\to\infty}\sup_{x\le (-c^*-\varepsilon)t}u(t,x;u_0)=\lim_{t\to\infty}\sup_{x\ge (c^*+\varepsilon)t}u(t,x;u_0)=0,
$$
and
$$
\liminf_{t\to\infty}\inf_{(-c^*+\varepsilon)t\le x\le (c^*-\varepsilon)t}u(t,x;u_0)>0.
$$
 Moreover, if $b>2\chi\mu$, then
$$\lim_{t\to\infty} \sup_{(-c^*+\varepsilon)t\le x\le (c^*-\varepsilon)t}|u(t,x;u_0)-p(t)|=0.$$
\end{itemize}
\end{tm}

Biologically, Theorem \ref{compact-support} shows that, when the chemotactic effect is weak relative to the self-limitation effect and the species is initially confined to a bounded region, three distinct propagation scenarios may occur. If $c>c^*$, then the species eventually becomes extinct. If $-c^*\leq c<c^*$, then the species persists and spreads along the shifting habitat into an increasingly large spatial region. If $c<-c^*$, then the species persists and spreads in both directions with asymptotic spreading speed $c^*$. Moreover, in the latter two cases, if $b>2\chi\mu$, then the population converges uniformly to the positive time-periodic state $p(t)$ in the corresponding interior spreading regions.

%Biologically, Theorem \ref{compact-support} implies that in the  case where the chemotactic sensitivity is small relative to the self-limitation rate of the species and the species initially lives in a bounded region, there are three possibilities:  if $c>c^{*}$, then the species will die out in the long run;  if $-c^{*}\leq c<c^{*}$, then the species will persist and spread along the shifting habitat into larger and larger regions; and if $c<-c^{*}$, the species will persist and spread towards both directions at the asymptotic spreading speed $c^{*}$.

Our second result concerns the vanishing, persistence, and spreading dynamics of solutions to \eqref{Keller-Segel-eq0} with front-like initial data.

\begin{tm}\label{right-support}
Suppose that {\bf (H1)} and {\bf (H2)} hold,
and let $u_0(x)$ be nonnegative, bounded, and $u_0(x)=0$ for $x\ll  -1$ and $\liminf_{x\to\infty}u_0(x)>0$. Then the following statements are valid:
\begin{itemize}
\item[(1)]
 If $c\geq-c^*$, then for every $\varepsilon>0$, there hold
$$\lim_{t\to\infty}\sup_{x\leq(c-\varepsilon)t}u(t,x;u_0)=0,  \quad \text{and}
\quad
\liminf_{t\to\infty}\inf_{x\ge  (c+\varepsilon)t}u(t,x;u_0)>0.$$
 Moreover, if $b>2\chi\mu$, then
$$\lim_{t\to\infty} \sup_{x\ge (c+\varepsilon)t}|u(t,x;u_0)-p(t)|=0.$$
\item[(2)]
If $c<-c^*$, then for every $\varepsilon>0$, there hold 
$$
\lim_{t\to\infty}\sup_{x\leq(-c^*-\varepsilon)t}u(t,x;u_0)=0,\quad \text{and}
\quad
\liminf_{t\to\infty}\inf_{x\ge  (-c^*+\varepsilon)t}u(t,x;u_0)>0.$$
 Moreover, if $b>2\chi\mu$, then
$$\lim_{t\to\infty} \sup_{x\ge (-c^*+\varepsilon)t}|u(t,x;u_0)-p(t)|=0.$$
\end{itemize}
\end{tm}

Biologically, Theorem \ref{right-support} shows that, when the chemotactic effect is weak relative to the self-limitation effect and the species initially occupies the right-hand side of the habitat, two distinct propagation scenarios may occur.  if $c\geq -c^{*}$, then the species persists and spreads along the shifting habitat.
If $c< -c^{*}$, then  the species persists and spreads leftward with asymptotic spreading speed $c^{*}$. Moreover, in both cases, if $b>2\chi\mu$, then the population converges uniformly to the positive time-periodic state $p(t)$ behind the corresponding propagation front.

\medskip

A positive solution $(u(t,x),v(t,x))$ of system \eqref{Keller-Segel-eq0} is  called a {\it time $\omega$-periodic forced wave solution} connecting $(0, 0)$ and $(p(t), \frac{\mu}{\lambda}p(t))$ if
 it is defined for all $t, x\in\R$, and can be written as $(u(t,x),v(t,x))=(U(t, x-ct),V(t, x-ct))$ for some unctions $U, V: \R\times\R\to\R$ satisfying 
$$
\lim_{z\to-\infty} U(t,z)=\lim_{z\to-\infty}V(t,z)=0 \quad {\rm uniformly} \,\ {\rm in}\,\ t\in\R,
$$ 
$$
\lim_{z\to\infty} |U(t,z)-p(t)|=\lim_{z\to\infty}|V(t,z)-\frac{\mu}{\lambda} p(t)|=0 \quad {\rm uniformly} \,\ {\rm in}\,\ t\in\R,
$$
and
$$U(t+\omega, z)=U(t, z), \quad\quad V(t+\omega, z)=V(t, z), \quad t,z\in\R.$$
 %It is clear that $(U(\cdot, \cdot),V(\cdot, \cdot))$ satisfies
% \begin{equation}\label{wave-eq}
%\begin{cases}
%U_t= U_{xx}+s^{'}(t)U_x-\chi(UV_x)_x+U(r(t, x)-bU),\quad t\in\R, \,\ x\in\R,\cr
%0 =V_{xx}-  \lambda V +\mu U,\quad t\in\R, \,\, x\in\R.
%\end{cases}
%\end{equation}

By  the definition of $c$, the function $s(t)-ct$ is  
$\omega$-periodic. Hence, defining $(\tilde U(t, x), \tilde V(t, x)):=(U(t, x+s(t)-ct), V(t, x+s(t)-ct))$ is also $\omega$-periodic in $t$. Moreover,
$(\tilde U(t, x-s(t)), \tilde V(t, x-s(t)))=(U(t, x-ct), V(t, x-ct))$.
%This gives rise to an equivalent form $(\tilde U(t, x-s(t)), \tilde V(t, x-s(t)))$ for  a forced traveling wave solution of system \eqref{Keller-Segel-eq0}.
Thus, a time $\omega$-periodic forced wave solution can equivalently be represented in the form $(\tilde U(t, x-s(t)), \tilde V(t, x-s(t)))$.

Our third result concerns the existence of time-periodic forced wave solutions.

\begin{tm}\label{forced-wave}
If $b>(1+\frac{\bar r}{\underline r})\chi\mu$ and 
$c>\frac{\chi\mu \bar r }{2\sqrt{\lambda}(b-\chi\mu)}-2\sqrt{\frac{1}{\omega}\int_{0}^{\omega}(r(t, \infty)-\frac{\chi\mu\bar r}{b-\chi\mu})dt}$, then system \eqref{Keller-Segel-eq0} has a time $\omega$-periodic forced wave solution $(u(t,x),v(t,x))=(U(t, x-ct),V(t, x-ct))$ connecting $(0, 0)$ and $(p(t), \frac{\mu}{\lambda}p(t))$.
\end{tm}

Our last result concerns the non-existence of time-periodic forced wave solutions.

\begin{tm}\label{nonexistence-forced-wave}
Suppose that ${\bf (H2)}$ holds. If  $c<-c^*$, then system \eqref{Keller-Segel-eq0} has no time $\omega$-periodic forced wave solution $(u(t,x),v(t,x))=(U(t, x-ct),V(t, x-ct))$ connecting $(0, 0)$ and $(p(t), \frac{\mu}{\lambda}p(t))$.
\end{tm}

Our findings reveal the combined effects of chemotaxis and a time-periodic shifting environment on species propagation. A major analytical difficulty arises from the lack of a comparison principle for the chemotaxis system, together with the time-periodic nature of the shifting environment. To overcome these difficulties, we reformulate the original system into an equivalent system and develop new analytical tools for studying the resulting problem.

The rest of this paper is organized as follows. In Section 2, we establish several preliminary lemmas that will be used in the proofs of our main results. In Section 3, we investigate the vanishing and spreading properties of solutions to system \eqref{Keller-Segel-eq0} and prove Theorems \ref{compact-support} and \ref{right-support}. In Section 4, we study the existence and nonexistence of forced wave solutions and prove Theorems \ref{forced-wave} and \ref{nonexistence-forced-wave}.
\section{Preliminaries}
In this section, we present some preliminary lemmas to be used in the proofs of the main theorems in later sections.

Observe that if $\hat r(t, \xi)=r(t, \xi+ct-s(t))$, then 
\begin{align*}
\hat r(t+\omega, \xi)&=r(t+\omega, \xi+c(t+\omega)-s(t+\omega))\cr
&=r(t, \xi+c(t+\omega)-s(t+\omega))\cr
&=r(t, \xi+ct-s(t))\cr
&=\hat r(t, \xi),
\end{align*}
where we used the $\omega$-periodicity of $r(t, x)$ in $t$ and the identity $s(t+\omega)=s(t)+c\omega$. Moreover, $\hat r(t, \xi)$ also satisfies assumption {\bf (H1)}. Since $\hat r(t, x-ct)=r(t, x-s(t))$, the term $r(t, x-s(t))$ in the first equation of  \eqref{Keller-Segel-eq0} can be equivalently rewritten as $\hat r(t,x-ct)$. For notational simplicity, we continue to denote $\hat r$ by $r$. Thus, \eqref{Keller-Segel-eq0} can be written as
\begin{equation}\label{Keller-Segel-eq0-1}
\begin{cases}
u_t=u_{xx}-\chi( u v_{x})_{x}+u(r(t, x-ct)-bu),\quad x\in\R\cr
0 =v_{xx}-  \lambda v +\mu u,\quad x\in\R.
\end{cases}
\end{equation}

From the second equation in \eqref{Keller-Segel-eq0-1}, we see that
$$
v_{xx}=\lambda v-\mu u,
$$
and hence, the first equation in \eqref{Keller-Segel-eq0-1} can be written as
$$
u_t=u_{xx}-\chi v_{x} u_{x}+ u\big( r(t, x-ct)-\chi \lambda v-(b-\chi\mu)u\big),\quad x\in\R.
$$
By the comparison principle, it follows that if $b>\chi\mu$, then for any $u_0\in C_{\rm unif}^b (\R)$ with $u_0\ge 0$,
$$
0\le u(t,x;u_0)\le { \max}\{\|u_0\|_\infty, \frac{\bar r}{b-\chi\mu}\}\quad \forall\,\ t\ge 0,\,\, x\in\R.
$$
%where $\bar r:=\max_{t\in [0,\omega]}r(t,\infty)$.
%Then by a priori estimates for parabolic equations, $v(t,x;u_0)$ and $v_x(t,x;u_0)$ are globally H\"older continuous in $t\ge \delta_0$ and $x\in\R$ for any $\delta_0>0$.

%{\color{blue} 
By the arguments similar to those for  \cite[Lemma 2.5]{SaShXu}, we have
the following result.%}

\begin{lem}
\label{new-lm2}
Assume $b>\chi\mu$. Then
for every $R\gg 1$, there are $C_R\gg 1$ and $\varepsilon_R>0$ such that for any $u_0\in C_{\rm unif}^b(\R)$ with $u_0\ge 0$, any
$x_0\in \R$, and any $t\ge 0$,  we have
 \begin{equation}\label{nnew-eq3}
\|\chi v_x(t,\cdot;u_0)\|_{L^{\infty}(B_{\frac{R}{2}}(x_0))} + \|\chi\lambda v(t,\cdot;u_0)\|_{L^{\infty}(B_{\frac{R}{2}}(x_0))}\leq C_{R}\|u(t,\cdot;u_0)\|_{L^{\infty}(B_{R}(x_0))}+\varepsilon_R M
\end{equation}
with $\lim_{R\to\infty}\varepsilon_R=0$, where $M:=\max\{\|u_0\|_{\infty},\frac{\bar r}{b-\chi\mu}\}$.
\end{lem}

%\begin{proof} It follows from the arguments of  \cite[Lemma 2.5]{SaShXu}.
%	{\color{red} revise?,   write a sentence in front of Lemma 2.1}
%\end{proof}

%{

%\begin{lem}
%\label{new-lm6}
%Suppose that $f: (0, \infty)\times\R\mapsto\R$ is a bounded vector field with $\||f|\|_{L^{\infty}(0,\infty)\times\R}\le A<\infty$.
%let $\Gamma(t,s,x,y)$ be the fundamental solution to
%\begin{equation}\label{nnew-eq01}
%\begin{cases}
%\Gamma_{t}+f\cdot\nabla_{x}\Gamma=\Delta_{x}\Gamma,\quad x\in\R, \, t>s\cr
%\Gamma(t=s,s,x,y)=\delta_{y}(x)
%\end{cases}
%\end{equation}
%for any $t>s\ge 0$ and $y\in\R$. Then
%$$
%\Gamma(t,s,x,y)=\frac{1}{\sqrt{4\pi(t-s)}}+\frac{A}{2}
%$$
%for all $t>s\ge 0$ and $(x,y)\in\R\times\R$, and
%$$
%\Gamma(t,s,x,y)=(\frac{1}{\sqrt{4\pi(t-s)}}+\frac{A\sqrt{t-s}}{\sqrt{4\pi}(|x-y|-A(t-s))})e^{-\frac{(|x-y|-A(t-s))^2}{4(t-s)}}
%$$
%if $|x-y|>A(t-s)$. Further, if $|x-y|>A(t-s)$ then
%$$
%\frac{e^{-\frac{(|x-y|-A(t-s))^2}{4(t-s)}}}{(16\pi(t-s))^{\frac{1}{2}}}\leq \Gamma(t,s,x,y).
%$$
%\begin{proof}
%This is \cite[Lemma 4.2]{FhCh}
%\end{proof}
%\end{lem}
%}

By slightly modifying the arguments of \cite[Lemma 2.2]{ShXu1}, we can
further prove the  following result.
\begin{lem}
\label{new-lm3}
Assume $b>\chi\mu$.  Then for every $p>1$, $t_0>0$, $s_0\geq 0$, $R>0$, and  $u_0\in C_{\rm unif}^b(\R)$ with $u_0\ge 0$, there is $C_{t_0,s_0,R,M,p}$ such that if $s\in[0,s_0], \ t\geq t_0,  \ |x-y|\leq R$, then
\begin{equation}\label{nnew-eq4}
u(t,x;u_0)\le C_{t_0,s_0,R,M,p}[u(t+s,y;u_0)]^{\frac{1}{p}}(M+1),
\end{equation}
where   $M:=\max\{\|u_0\|_{\infty},\frac{\bar r}{b-\chi\mu}\}$.
\end{lem}

%\begin{proof}
%It  can be proved by slightly  modifying  the arguments  of \cite[Lemma 2.2]{ShXu1}. 	{\color{red} revise?,   write a sentence in front of Lemma 2.2}
%\end{proof}

By \eqref{nnew-eq3} and \eqref{nnew-eq4} with $p>1$, $s_0=0$ and $t_0=1$, we have
\begin{align}
\label{E2-1}
\left|\chi v_x(t,x;u_0)\right|+\chi \lambda v(t,x;u_0)&\le %C_{R,p} \big(u(t,x;u_0)\big)^{\frac{1}{p}}+\varepsilon_R M\nonumber\\
%&= 
C_{R,p} \big(u(t,x;u_0)\big)^{\frac{1}{p}}+\varepsilon_R M \quad \forall\,\, t\ge 1,\,\, x\in\R,
\end{align}
where  $C_{R,p}=C_R\cdot  C_{1,0,R,M,p}\cdot (M+1) (>0)$.

\smallskip

For every $u\in C(\R, C^b_{\rm unif}(\R))$, let
\begin{equation}\label{psi-definition}
\Psi(t, x;u)=\mu\int_{0}^{\infty}\int_{\R}\frac{e^{-\lambda s}e^{-\frac{|y-x|^2}{4s}}}{\sqrt{4\pi s}}u(t, y)dyds.
\end{equation}
It is not difficult to prove that $\Psi(\cdot, \cdot; u)$ solves the elliptic equation
$$
\partial_{xx}\Psi(t, x;u)-\lambda\Psi(t, x;u)+\mu u(t,x)=0.
$$

\begin{lem}
\label{new-lm1}
{\sc (\cite[lemma 2.1]{SaShXu})} There hold
\begin{equation}\label{psi-definition-eq2}
\Psi(t, x;u)=\frac{\mu}{2\sqrt{\lambda}}\int_{\R}e^{-\sqrt{\lambda}|x-y|}u(t, y)dy
\end{equation}
and
\begin{equation}\label{space-derivative-of-psi-1}
\frac{\partial}{\partial x}\Psi(t, x;u)=-\frac{\mu}{2}e^{-\sqrt{\lambda }x}\int_{-\infty}^xe^{\sqrt{\lambda}y}u(t, y)dy + \frac{\mu}{2}e^{\sqrt{\lambda}x}\int_{x}^{\infty}e^{-\sqrt{\lambda}y}u(t, y)dy.
\end{equation}
\end{lem}

By a straightforward computation, we have the following estimate.
\begin{lem}\label{V-estimat}
Suppose that $b>\chi\mu$. For every $u\in C^b_{\rm unif}(\R)$, $0\leq u(x)\leq \frac{\bar r}{b-\chi\mu}$, it holds that
$$
\Psi(t, x;u)\leq \frac{\mu \bar r}{\lambda(b-\chi\mu)} \quad \forall\,\ t\in\R, \,\ x\in\R
\quad {\rm and}\quad
\p_{x}\Psi(t,x;u)\leq \frac{\mu \bar r}{2\sqrt{\lambda}(b-\chi\mu)}  \quad \forall\,\ t\in\R, \,\ x\in\R.
$$
\end{lem}

Let 
\begin{equation}\label{psi-eq}
\psi(t)=e^{\int_{0}^{t}(r(s,\infty)-r^{*})ds}.
\end{equation}
It follows that $\psi(t)$ is a positive $\omega$-periodic function and satisfies 
$$
\psi^{'}=(r(t,\infty)-r^{*})\psi.
$$

The following result  can be proved by the arguments similar to those for \cite[Lemma 2.2]{SaShXu}.

\begin{lem}\label{lem-001}
For every $u\in C(\R, C^b_{\rm unif}(\R))$, $u(x)\ge 0$, it holds that
\begin{equation}\label{estimate-on-space-derivative-1}
\left| \p_{x}\Psi(t, x;u)\right|\leq \sqrt{\lambda}\Psi(t, x;u),\ \quad  \forall\ t\in\R, \,\, x\in\R.
\end{equation}
In particular, for every solution $(u(t,x;u_0),v(t,x;u_0))$ of \eqref{Keller-Segel-eq0} with $u_0\geq 0$ it holds that
\begin{equation*}
|v_{x}(t,x;u_0)|\leq \sqrt{\lambda}v(t,x;u_0),\quad \forall\ x\in\R,\,\ t>0.
\end{equation*}
Furthermore,  if
\begin{equation}
\frac{(\kappa-\sqrt{\lambda})_{+}}{(\kappa +\sqrt{\lambda})}\leq \frac{2(b-\chi\mu)}{\chi\mu},
\end{equation} it holds that
\begin{equation}\label{estimate-on-space-derivative-2}
\chi\kappa\Psi_x(t,x;u)-\chi\lambda\Psi(t,x;u)-(b-\chi\mu)Me^{-\kappa x}\psi(t)\leq 0, \quad \forall\,\ x\in\R, \,\ t\geq 0
\end{equation}
whenever $0\leq u(t, x)\leq Me^{-\kappa x}\psi(t)$ for some positive real numbers $\kappa>0$ and $M>0$.
\end{lem}

%\begin{proof}
%It can be proved by similar arguments used in \cite[Lemma 2.2]{SaShXu}.
%	{\color{red} revise?,   write a sentence in front of this lemma}
%\end{proof}

\begin{lem}
\label{lem-001-1}
 Assume that $b>\chi\mu$. Let $c_{\kappa}=\frac{\kappa^2+r^*}{\kappa}$ with $ 0<\kappa\leq \sqrt{r^*}$ satisfying
$$\frac{(\kappa-\sqrt{\lambda})_{+}}{(\kappa +\sqrt{\lambda})}\leq \frac{2(b-\chi\mu)}{\chi\mu}.$$ Then the  following statements are valid:
\begin{enumerate}
\item[(i)]
 For any $u_0\ge 0$ with  nonempty compact support
 and  any $M\gg \frac{\bar r}{b-\chi\mu}$  satisfying
\begin{equation}\label{u_0-cond1}
\max\{u_0(x),u_0(-x)\}\leq \min\{M,Me^{-\kappa x}\},\quad \forall\ x\in\R,
\end{equation}
there holds
\begin{equation}
\label{new-eq1}
u(t,x;u_0)\leq M e^{-\kappa(|x|-c_{\kappa}t)}\max_{t\in[0,\omega]}\psi(t), \quad\forall\ x\in\R,\ t\geq 0.
\end{equation}

\item[(ii)]
 For any $u_0\in C^b_{\rm unif}(\R)$, $u_0(x)\geq 0 $,   and  any  $M\gg \frac{\bar r}{b-\chi\mu}$  satisfying
\begin{equation}\label{u_0-cond3}
u_0(x)\leq \min\{M,Me^{\kappa x}\},\quad \forall\ x\in\R,
\end{equation}
there holds
\begin{equation}
\label{new-eq1-3}
u(t,x;u_0)\leq M e^{\kappa(x+c_{\kappa}t)}\max_{t\in[0,\omega]}\psi(t), \quad\forall\ x\in\R,\ t\geq 0.
\end{equation}
\end{enumerate}
\end{lem}

\begin{proof}
We shall only prove $(i)$ since $(ii)$ can be proved by the similar arguments.

Let $U^{+}(t,x)=\min\{Me^{-\kappa x}\psi(t), M\}$. For any given  $T>0$, let
$$
\mathcal{E}^T:=\{u\in C([0,T], C^b_{\rm unif}(\R)) : u(0,\cdot)=u_0(\cdot)\ \text{and}\ 0\leq u(t,x)\leq {U}^{+}(t, x)\ \forall\ x\in\R, \ 0\leq t\leq T\}.
$$
For every $u\in\mathcal{E}^T$, let $\Phi(t,x;u)$ denote the solution of
\begin{equation*}
\begin{cases}
\Phi_t=\mathcal{A}_u(\Phi), \ x\in\R,\,\ 0<t\leq T,\cr
\Phi(0,x; u)=u_0(x),\quad x\in\R.
\end{cases}
\end{equation*}
where $$\mathcal{A}_u(\Phi):=\Phi_{xx}+(c_\kappa-\chi\Psi_x(\cdot,\cdot;u))\Phi_x+(r(t, x+c_{\kappa}t-ct)-\chi\lambda\Psi(\cdot,\cdot;u)-(b-\chi\mu)\Phi)\Phi$$
and $\Psi(t, x;u)$ is the solution of
$$
\Psi_{xx}-\lambda \Psi+\mu u=0,\quad x\in\R, \,\ t\geq 0.
$$
Observe that
\begin{align*}\label{w-eq0}
&\mathcal{A}_u(Me^{-\kappa x}\psi)-(Me^{-\kappa x}\psi)_{t}\cr
%&=\left(\kappa^2-\kappa c_{\kappa}+\chi\kappa\Psi_x(t,x;u)+r(t,x+c_{\kappa}t-ct)-\chi\lambda\Psi(t,x;u)-(b-\chi\mu)Me^{-\kappa x}\psi-\frac{\psi^{'}}{\psi}\right)Me^{-\kappa x}\psi\cr
&\leq \left(\kappa^2-\kappa c_{\kappa}+\chi\kappa\Psi_x(t,x;u)+r(t, \infty)-\chi\lambda\Psi(t,x;u)-(b-\chi\mu)Me^{-\kappa x}\psi-\frac{\psi^{'}}{\psi}\right)Me^{-\kappa x}\psi\cr
&= \left(\chi\kappa\Psi_x(t,x;u)-\chi\lambda\Psi(t,x;u)-(b-\chi\mu)Me^{-\kappa x}\psi\right)Me^{-\kappa x}\psi\cr
%&\leq \left(\chi\kappa\Psi_x(t,x;u)-\chi\lambda\Psi(t,x;u)-(b-\chi\mu)Me^{-\kappa x}\min_{t\in [0,\omega]}\psi(t)\right)Me^{-\kappa x}\psi.\cr
\end{align*}
It then follows from Lemma \ref{lem-001} that
$$
\mathcal{A}_u(Me^{-\kappa x}\psi)\leq 0.
$$
Since $u_0(x)\leq Me^{-\kappa x}$ for every $x\in\R$, the comparison principle implies that
$$
\Phi(t,x;u)\leq Me^{-\kappa x}\psi(t), \quad \forall\ x\in\R,\,\ 0\leq t\leq T.
$$
On the other hand, since $M\geq \frac{\bar r}{b-\chi\mu}$, we have that
$$
\mathcal{A}_u(M)=(r(t, x+c_{\kappa}t-ct)-\chi\lambda \Psi(t,x;u)-(b-\chi\mu)M)M\leq 0.
$$
Since $u_0(x)\leq M$ for every $x\in\R$, by comparison principle again, we conclude that
$$
\Phi(t,x;u)\leq M, \quad \forall\ x\in\R, \,\ 0\leq t\leq T.
$$
Therefore, we obtain
$$
0\leq \Phi(t,x;u)\leq U^+(t, x),\quad\forall\ x\in\R,\ 0\leq t\leq T,\ \ u\in\mathcal{E}^T.
$$

By the arguments similar to those for  \cite[Theorem 3.1 ]{SaSh2}, it can be shown that the function  $\Phi : \mathcal{E}^T\ni u \mapsto \Phi(\cdot,\cdot;                                                                                                                                                                          u)\in \mathcal{E}^T $ is continuous and compact in the compact open topology. Hence, by the Schauder fixed point theorem, there is $u^*\in \mathcal{E}^T$ such that $\Phi(u^*)=u^*$. Let $v^{*}(t, x)=\Psi(t, x; u^{*})$. Note that $(u^*(t,x- c_{\kappa}t),v^*(t,x- c_{\kappa}t))$ is a solution of \eqref{Keller-Segel-eq0-1} with $u^*(0,x)=u_0(x)$ for every $x\in\R$. Hence, by uniqueness of the solution to \eqref{Keller-Segel-eq0-1}, we conclude that
$$
u^*(t,x-c_\kappa t)=u(t,x;u_0),\quad \forall\ x\in\R, \ 0\leq t\leq T.
$$
Hence $u(t,x;u_0)\in \mathcal{E}^T$. Since $T$ was arbitrary chosen, we obtain that
$$
u(t,x;u_0)\leq U^{+}(t, x-c_{\kappa}t)\leq M e^{-\kappa(x-c_{\kappa}t)}\max_{t\in[0,\omega]}\psi(t), \quad\forall\ x\in\R,\ t\geq 0.
$$
Similar arguments as in the above yield that
$$
u(t,-x;u_0)\leq U^{+}(t, x-c_{\kappa}t)\leq M e^{-\kappa(x-c_{\kappa}t)}\max_{t\in[0,\omega]}\psi(t), \quad\forall\ x\in\R,\ t\geq 0.
$$
It then follows that 
$$
u(t,x;u_0)\leq M e^{-\kappa(|x|-c_{\kappa}t)}\max_{t\in[0,\omega]}\psi(t), \quad\forall\ x\in\R,\ t\geq 0.
$$
This completes the proof.
\end{proof}

%Consider
%\begin{equation}\label{Le2.6-eq1}
%\begin{cases}
%u_t=u_{xx}-\chi(uv_x)_x +u(r^*-bu),\quad x\in\R\cr
%0=v_{xx}- \lambda v+\mu u,\quad x\in\R.
%\end{cases}
%\end{equation}

%\begin{lem}
%\label{lem-001-4}
%Assume that $b>2\chi\mu$. Then for any $u_0\in C_{\rm unif}^b(\R)$ with $\inf_{x\in\R}u_0(x)>0$,
%$$
%\lim_{t\to\infty} \|u(t,\cdot;u_0)-\frac{r^*}{b}\|_\infty=0,
%$$
%where $(u(t,x;u_0),v(t,x;u_0))$ is the solution of \eqref{Le2.6-eq1} with $u(0,x;u_0)=u_0(x)$.
%\end{lem}

%\begin{proof}
%It follows from \cite[Theorem 1.8]{SaSh1}.
%\end{proof}

\section{Persistence and spreading speeds}

In this section, we study the spreading properties of \eqref{Keller-Segel-eq0} and prove Theorems \ref{compact-support} and \ref{right-support}.  Throughout this section, we assume that {\bf (H1)} and {\bf (H2)} hold.  We start with two lemmas.

For any given $\bar c$, let $u(t,x)=\tilde u(t,x-\bar c t)$ and $v(t,x)=\tilde v(t,x-\bar c t)$. Then \eqref{Keller-Segel-eq0-1} becomes
\begin{equation}\label{new-Keller-Segel-eq}
\begin{cases}
\tilde u_t=\tilde u_{xx}+\bar c \tilde u_x-\chi (\tilde  u \tilde  v_{x})_{x}+\tilde u(r(t, x-(c-\bar c)t)-b\tilde u),\quad x\in\R\cr
0 =\tilde v_{xx}-  \lambda \tilde v +\mu \tilde u,\quad x\in\R.
\end{cases}
\end{equation}
In the following,  $(u(t,x;u_0),v(t,x;u_0))$ denotes the solution of \eqref{Keller-Segel-eq0-1} (also the solution of \eqref{Keller-Segel-eq0}) with
$u(0,x;u_0)=u_0(x)$, and $(\tilde u(t,x;u_0),\tilde v(t,x;u_0))$ denotes the solution of \eqref{new-Keller-Segel-eq} with
$\tilde u(0,x;u_0)=u_0(x)$.

For any given $0<\epsilon<2\sqrt{r^*}$, let 

\begin{equation}
\label{l-eq}
l=\frac{2\pi}{(\epsilon\sqrt{r^*})^{\frac{1}{2}}}
\end{equation}
%\begin{equation}
%\label{bar-r-eq}
%4\bar r -\bar c^2\ge \epsilon \sqrt{r^*}\quad \forall\,\,   -2\sqrt{r^*}+\epsilon \le  \bar c  \le  2\sqrt {r^*}-\epsilon.
%\end{equation}
% Let
%\begin{equation}
%\label{l-eq}
%l=\frac{2\pi}{(\epsilon\sqrt{r^*})^{\frac{1}{2}}}
%\end{equation}
and
\begin{equation}
\label{lambda-eq}
\lambda(\bar c)=\frac{4r^{*}-\bar c^2-\frac{\pi^2}{l^2}}{4}.
\end{equation}
It follows that  $\lambda(\bar c)>  \frac{7\epsilon\sqrt {r^*}}{16}>0$ for any $ -2\sqrt{r^*}+\epsilon  \le \bar c \le   2\sqrt {r^*}-\epsilon$. Further, a straightforward computation gives rise to the following result.

\begin{lem}
\label{lem-001-2}
Given $0<\epsilon<2\sqrt{r^*}$, let $l$ be as in \eqref{l-eq}. Then
for any $ -2\sqrt{r^*}+\epsilon \leq \bar c \leq  2\sqrt {r^*}-\epsilon$, $\lambda(\bar c)$, which is defined as in \eqref{lambda-eq}, is the principal eigenvalue of
\begin{equation}
\label{ev-eq1}
\begin{cases}
-\phi_{t}+\phi_{xx}+\bar c\phi_x+r(t, \infty) \phi=\lambda \phi,\quad -l<x<l\cr
\phi(-l)=\phi(l)=0,
\end{cases}
\end{equation}
and $\phi(t, x;\bar c)=\psi(t)e^{-\frac{\bar c}{2}x}{ \cos\frac{\pi}{2l}x}$ is a corresponding positive eigenfunction where $\psi(t)$ is defined as in \eqref{psi-eq}.
\end{lem}

Let $\lambda(\bar c, R)$ be the principle eigenvalue of \eqref{ev-eq1} with $r(t,\infty)$ being replaced by $r(t,\infty)-2\varepsilon_R M$. Then we have

$$
\lambda(\bar c, R)=\lambda(\bar c)-2\varepsilon_R M,
$$
and

$$
\lim_{R\to\infty} \lambda(\bar c, R)=\lambda(\bar c)
$$

Consider
\begin{equation}
\label{aux-new-eq01}
\begin{cases}
u_t=u_{xx}+\bar c u_x-A(t,x)u_x+u (r(t,\infty) -2 \varepsilon_R M-C_{R,p} u^{\frac{1}{p}}-(\bar b-\chi\mu)u), \quad -l<x<l\cr
u(t,-l)=u(t,l)=0,
\end{cases}
\end{equation}
where $R\gg 1$ is such that $ \lambda(\bar c, R)>0$ for any $ -2\sqrt{r^*}+\epsilon  \le \bar c \le   2\sqrt {r^*}-\epsilon$,  and 
$A(t,x)$ is globally H\"older continuous in $t\in\R$ and
$x\in [-l,l]$ with H\"older exponent $0<\alpha<1$ and $\|A(\cdot,\cdot)\|_\infty<\infty$.

\begin{lem}
\label{lem-001-3} For given $0<\epsilon<2\sqrt{r^*}$, let $l$ be as in \eqref{l-eq}.
There is $\eta>0$ such that for any $A(\cdot,\cdot)$ with $\|A(\cdot,\cdot)\|_\infty<\eta$,  any $\bar c\in {[ -2\sqrt{r^*}+\epsilon} ,2\sqrt{r^*}-\epsilon { ]}$, and
any $\bar b \in (\chi\mu,\infty)$,
\eqref{aux-new-eq01} has a positive bounded entire solution
$u(t,x;\bar c,\bar b, A)$ satisfying that
 \begin{equation}
 \label{aux-new-new-eq0}
 \inf_{-l+\delta\le x\le l-\delta,t\in\R, \bar c\in {[ -2\sqrt{r^*}+\epsilon}, 2\sqrt{r^*}-\epsilon { ]}} u(t,x;\bar c,\bar b,A)>0\quad \forall \,\, 0<\delta <l.
 \end{equation}
\end{lem}

\begin{proof}

First of all, there is $\bar\varepsilon>0$ such that
$$
\lambda(\bar c, R)-\bar\varepsilon>0.%\quad \forall\,\, R\ge R^*, \,\, 0<\bar\varepsilon\leq \bar\varepsilon^{*}.
$$
%Let $R\ge R^*$ and $0<\bar\varepsilon\leq \bar\varepsilon^{*}$ be fixed, 
Define $\lambda(\bar c, R, \bar \varepsilon):=\lambda(\bar c, R)-\bar\varepsilon$, consider
\begin{equation}
\label{aux-new-eq02}
\begin{cases}
\tilde u_t=\tilde u_{xx}+\bar c\tilde  u_x-A(t,x)\tilde u_x+(r(t, \infty)-2\varepsilon_R M-\bar\varepsilon) \tilde u , \quad  -l<x<l\cr
\tilde u(t,-l)=\tilde u(t,l)=0.
\end{cases}
\end{equation}
For any given $n\in\Z$, let $\tilde u(t,x;n\omega,\phi(0, \cdot;\bar c),A)$ solve \eqref{aux-new-eq02} with $\tilde u(n\omega,x; n\omega,\phi(0, \cdot;\bar c),A)= \phi(0, x;\bar c)$. We claim that
\begin{equation}
\label{aux-new-new-eq2}
\lim_{\|A(\cdot,\cdot)\|_{\infty}\to 0} \|\tilde u(t,\cdot;n\omega,\phi(0, \cdot;\bar c),A)-e^{\lambda(\bar c, R, \bar \varepsilon)(t-n\omega)}\phi(t-n\omega, \cdot;\bar c)\|_{C^1([-l,l])}=0
\end{equation}
uniformly in $t\in [n\omega, n\omega+\omega]$, $n\in\Z$, and  $\bar c\in {[ -2\sqrt{r^*}+\epsilon}, 2\sqrt{r^*}-\epsilon { ]}$.

Indeed, let $\tilde{\tilde u}(t, x; 0, \phi(0, \cdot;\bar c), A, n)=\tilde u(t+n\omega,x;n\omega,\phi(0, \cdot;\bar c),A)$ for $t\geq 0$, $-l\leq x\leq l$.  Then $\tilde{\tilde u}$ satisfies
\begin{equation}
\label{aux-new-eq02-double-tilde}
\begin{cases}
\tilde{\tilde u}_t=\tilde{\tilde u}_{xx}+\bar c\tilde{\tilde u}_x-A(t+n\omega,x)\tilde {\tilde u}_x+(r(t, \infty)-2\varepsilon_R M-\bar\varepsilon) \tilde{\tilde u} , \quad  t>0,\,\  -l<x<l\cr
\tilde{\tilde u}(t,-l)=\tilde{\tilde u}(t,l)=0, \quad t>0\cr
\tilde{\tilde u}(0, x)=\phi(0, \cdot;\bar c), \quad -l<x<l.
\end{cases}
\end{equation}
To prove \eqref{aux-new-new-eq2}, it is sufficient to prove
\begin{equation}
\label{aux-new-new-eq2-1}
\lim_{\|A(\cdot,\cdot)\|_{\infty}\to 0} \|\tilde{\tilde u}(t,\cdot;0,\phi(0, \cdot;\bar c),A, n)-e^{\lambda(\bar c, R, \bar \varepsilon)t}\phi(t, \cdot;\bar c)\|_{C^1([-l,l])}=0.
\end{equation}
uniformly in $t\in [0, \omega]$, $n\in\Z$, and  $\bar c\in {[ -2\sqrt{r^*}+\epsilon} ,2\sqrt{r^*}-\epsilon { ]}$.

Assume, by contradiction, that  there exist $\epsilon_0>0$, $t_{n}\in [0,\omega]$, $m_{n}\in\Z$, $\bar c_{n}\in [-2\sqrt{r^*}+\epsilon ,2\sqrt{r^*}-\epsilon]$ and $A_{n}(\cdot,\cdot)$ with $\|A_{n}(\cdot,\cdot)\|_{\infty}\to 0$ as $n\to\infty$ such that
\begin{equation}\label{tildetilde-contra-1}
 \|\tilde{\tilde u}(t_{n},\cdot;0,\phi(0, \cdot;\bar c_{n}),A_{n}, m_{n})-e^{\lambda(\bar c_{n}, R, \bar \varepsilon)t_{n}}\phi(t_{n}, \cdot;\bar c_{n})\|_{C^1([-l,l])}\geq \epsilon_{0}.
\end{equation}
Let $u_{n}(t,x)=\tilde{\tilde u}(t+t_{n}, x; 0, \phi(0, \cdot;\bar c_{n}), A_{n}, m_{n})$. Without loss of generality, we may assume that $t_{n}\to t^*$, $\bar c_{n}\to c^*$, $A_{n}(t+t_{n}+m_{n}\omega, x)\to A^{*}(t,x)$, $u_{n}(t,x)\to u^*(t,x)$ and $\p_{x}u_{n}(t,x)\to \p_{x}u^*(t,x)$ locally uniformly in $t\in(-t^*, \infty)$ and $x\in[-l, l]$ as $n\to\infty$. Since $\lim_{n\to\infty}\|A_{n}(\cdot,\cdot)\|_{\infty}=0$, we can conclude that $A^{*}(t, x)\equiv 0$. It then follows that $u^*(t,x)$ satisfies that
$$
u^{*}_{t}=u^{*}_{xx}+c^{*}u^{*}_{x}+(r(t+t^{*}, \infty)-2\varepsilon_R M-\bar\varepsilon)u^{*}, \quad t>-t^{*},\,\ -l<x<l.
$$
Next, we claim that $\lim_{t\to-t^{*}}u^{*}(t,x)=\phi(0,x; c^{*})$ uniformly in $x\in[-l,l]$.

Indeed, letting $w_{n}(t,x)=u_{n}(t-t_{n}, x)$ for $t\geq 0$, $-l\leq x\leq l$, we see that $w_{n}(t,x)$ satisfies
\begin{equation}
\label{aux-new-eq02-w}
\begin{cases}
(w_{n})_t=(w_{n})_{xx}+\bar c_{n}(w_{n})_x-A_{n}(t+m_{n}\omega,x)(w_{n})_x\cr
\quad \quad \quad \quad +(r(t, \infty)-2\varepsilon_R M-\bar\varepsilon) w_{n} , \quad t>0,\quad -l<x<l\cr
w_{n}(t,-l)=w_{n}(t,l)=0, \quad t>0\cr
w_{n}(0, x)=\phi(0, x;\bar c_{n}), \quad -l<x<l.
\end{cases}
\end{equation}
By the continuity of solution with respect to initial function and coefficients, it follows that $w_{n}(t,x)\to w(t,x)$ as $n\to\infty$ locally uniformly in $t\in [0, \infty)$ and $x\in [-l, l]$. Moreover, $w(t, x)=u^{*}(t-t^{*}, x)$ for $t>0$. In addition, 
$$
\lim_{t\to -t^{*}}u^{*}(t,x)=\lim_{t\to 0^{+}}u^{*}(t-t^{*}, x)=\lim_{t\to 0^{+}}w(t,x)=\phi(0, x; c^{*})
$$
uniformly in $x\in[-l,l]$. Hence, $u^{*}(t,x)$ can be extended to a continuous function on $[-t^*, \infty)$. It then follows that $$u^{*}(t,x)=e^{\lambda(c^{*}, R, \bar\varepsilon)(t+t^{*})}\phi(t+t^{*}, x; c^{*}).$$ In particular, $u^{*}(0,x)=e^{\lambda(c^{*}, R, \bar\varepsilon)t^{*}}\phi(t^{*}, x; c^{*})$ and $u^{*}_{x}(0,x)=e^{\lambda(c^{*}, R, \bar\varepsilon)t^{*}}\phi_{x}(t^{*}, x; c^{*})$. But, letting $n\to\infty$ in \eqref{tildetilde-contra-1}, we have 
\begin{equation*}\label{tildetilde-contra}
 \|u^{*}(0, \cdot)-e^{\lambda(c^{*}, R, \bar \varepsilon)t^{*}}\phi(t^{*}, \cdot;\bar c^{*})\|_{C^1([-l,l])}\geq \epsilon_{0},
\end{equation*}
which is a contradiction. This proves \eqref{aux-new-new-eq2-1},
and hence,  \eqref{aux-new-new-eq2} holds true.

By \eqref{aux-new-new-eq2}, there is $\eta>0$ such that for any $A(\cdot,\cdot)$ with $\|A(\cdot,\cdot)\|_\infty<\eta$, any
 $\bar c\in {[ -2\sqrt{r^*}+\epsilon} ,2\sqrt{r^*}-\epsilon { ]}$,
\begin{equation}
\label{aux-new-new-eq3}
 \tilde u(n\omega+\omega,x;n\omega,\phi(0, \cdot;\bar c),A)\ge  e^{\frac{\lambda(\bar c, R, \bar \varepsilon)}{2}\omega}\phi(\omega, x;\bar c), \,\,\,\forall\,  -l\le x\le l,\,\, n\in\Z.
\end{equation}

 Next, suppose that $\|A(\cdot,\cdot)\|_\infty<\eta$.
 Let $u(t,x;n\omega,\sigma \phi(0, \cdot;\bar c))$ be the solution of \eqref{aux-new-eq01} with
 $u(n\omega,x;n\omega,\sigma\phi(0, \cdot;\bar c))=\sigma\phi(0, x;\bar c)$.
We now claim that
$$
\lim_{\sigma\to 0} \sup_{t\in [n\omega, n\omega+\omega], -l\le x\le l} u(t,x;n\omega,\sigma\phi(0, \cdot;\bar c))=0
$$
uniformly in $n\in\Z$ and in  $\bar c\in {[ -2\sqrt{r^*}+\epsilon} ,2\sqrt{r^*}-\epsilon { ]}$.
Indeed, let $\tilde u(t,x; 0, \sigma\phi(0, \cdot; \bar c), n, \bar c)=u(t+n\omega,x;n\omega,\sigma\phi(0, \cdot;\bar c))$ for $t\geq 0$ and $-l\leq x\leq l$. Then $\tilde u(t,x; 0, \sigma\phi(0, \cdot; \bar c), n, \bar c)$ satisfies
\begin{equation}
\label{aux-new-eq01-tilde-ini}
\begin{cases}
u_t=u_{xx}+\bar c u_x-A(t+n\omega,x)u_x\cr
\quad \quad +u (r(t,\infty) -2 \varepsilon_R M-C_{R,p} u^{\frac{1}{p}}-(\bar b-\chi\mu)u), \quad t>0, \,\, -l<x<l\cr
u(t,-l)=u(t,l)=0,\quad t>0 \cr
u(0, x)=\sigma\phi(0,x;\bar c), \quad -l<x<l.
\end{cases}
\end{equation}
To prove the claim, it suffices to show that
\begin{equation}\label{aux-transfer-1}
\lim_{\sigma\to 0} \sup_{t\in [0, \omega], -l\le x\le l} \tilde u(t,x; 0,\sigma\phi(0, \cdot;\bar c), n, \bar c)=0
\end{equation}
uniformly in $n\in\Z$ and in  $\bar c\in {[ -2\sqrt{r^*}+\epsilon} ,2\sqrt{r^*}-\epsilon { ]}$. We prove it by contradiction. If this is not true, then there are $\epsilon_{0}>0$, $m_{n}\in\Z$, $\bar c_{n}\in {[ -2\sqrt{r^*}+\epsilon} ,2\sqrt{r^*}-\epsilon { ]}$ and $\sigma_{n}\to 0$ such that 
\begin{equation}\label{tilde-u-contra}
\sup_{t\in [0, \omega], -l\le x\le l} \tilde u(t,x; 0,\sigma_{n}\phi(0, \cdot;\bar c_{n}), m_{n}, \bar c_{n})\geq \epsilon_0.
\end{equation}
Consider $\tilde u_{n}(t,x)=\tilde u(t,x; 0,\sigma_{n}\phi(0, \cdot;\bar c_{n}), m_{n}, \bar c_{n})$. Then for every $n$, $\tilde u_{n}(t,x)$ solves
\begin{equation}
\label{aux-new-eq01-tilde-ini-1}
\begin{cases}
\p_{t}\tilde u_{n}=\p_{xx}\tilde u_{n}+\bar c_{n} \p_{x}\tilde u_{n}-A(t+m_{n}\omega,x)\p_{x}\tilde u_{n}\cr
\quad\quad\quad+\tilde u_{n} (r(t,\infty) -2 \varepsilon_R M-C_{R,p} \tilde u_{n}^{\frac{1}{p}}-(\bar b-\chi\mu)\tilde u_{n}), \,\, t>0, \,-l<x<l\cr
\tilde u_{n}(t,-l)=\tilde u_{n}(t,l)=0, \quad t>0 \cr
\tilde u_{n}(0, x)=\sigma_{n}\phi(0,x;\bar c_{n}). \quad -l<x<l.
\end{cases}
\end{equation}
Without loss of generality, we may assume that $\bar c_{n}\to c^*$, $A(t+m_{n}\omega,x)\to \tilde A(t,x)$ and $\tilde u_{n}(t,x)\to \tilde u(t,x)$ uniformly in $t\in [0, \omega]$ and $x\in [-l, l]$ as $n\to\infty$. Since $\sigma_{n}\phi(0,x;\bar c_{n})\to 0$ uniformly in $x\in [-l, l]$, we then have $\tilde u(t, x)\equiv 0$ for $t\in [0, \omega]$ and $x\in [-l, l]$. But by \eqref{tilde-u-contra}, $0\geq \epsilon_{0}>0$, which is a contradiction.
Thus, there is $\sigma_0>0$ such that for any $0<\sigma\le \sigma_0$, and any $\bar c\in {[ -2\sqrt{r^*}+\epsilon} ,2\sqrt{r^*}-\epsilon { ]}$,
$$
C_{R,p}  u^{\frac{1}{p}}(t,x; n\omega,\sigma\phi(0, \cdot;\bar c))+(\bar b-\chi\mu)u (t,x;n\omega,\sigma\phi(0, \cdot;\bar c))<\bar\varepsilon, \,\,
\forall\, t\in [n\omega, n\omega+\omega],\,  x\in [-l, l],\, n\in\Z.
$$
This together with the comparison principle implies that for any $0<\sigma \le \sigma_0$ and any $\bar c\in {[ -2\sqrt{r^*}+\epsilon} ,2\sqrt{r^*}-\epsilon { ]}$,
\begin{equation}
\label{aux-new-new-eq1}
u(t,x; n\omega,\sigma \phi(0, \cdot;\bar c))\ge \sigma \tilde u(t,x;n\omega,\phi(0, \cdot;\bar c),A),\,\,\,\forall\, t\in [n\omega, n\omega+\omega],\,\, -l\le x\le l,\,\, n\in\Z.
\end{equation}
Then \eqref{aux-new-new-eq3} implies that 
\begin{equation}
\label{aux-new-new-eq3-1}
 u(n\omega+\omega,x; n\omega,\sigma \phi(0, \cdot;\bar c))  \ge \sigma  e^{\frac{\lambda(\bar c, R, \bar \varepsilon)}{2}\omega}\phi(\omega, x;\bar c),\,\,\,\forall\,  -l\le x\le l,\,\, n\in\Z
\end{equation}
for any $0<\sigma\le \sigma_0$ and $\bar c\in {[ -2\sqrt{r^*}+\epsilon} ,2\sqrt{r^*}-\epsilon { ]}$.

Replacing $n$ with $-(n+1)$ in \eqref{aux-new-new-eq3-1}, since $\phi(\omega, x;\bar c)=\phi(0, x;\bar c)$, we obtain that 
\begin{equation}
 u(-n\omega,x;-(n+1)\omega,\sigma \phi(0, \cdot;\bar c))  \ge \sigma  \phi(0, x;\bar c),\,\,\,\forall\,  -l\le x\le l,\,\ n\in\Z
\end{equation}
for any $0<\sigma\le \sigma_0$ and $\bar c\in {[ -2\sqrt{r^*}+\epsilon} ,2\sqrt{r^*}-\epsilon { ]}$.
Then the comparison principle implies that 
$$
u(t,x;-(n+1)\omega,\sigma \phi(0, \cdot;\bar c))  \ge u(t, x; -n\omega, \sigma \phi(0, x;\bar c)),\,\,\,\forall\, t\geq -n\omega, -l\le x\le l.
$$
Let $u_n(t,x)=u(t,x;-n\omega,\sigma\phi(0, \cdot;\bar c))$.
Then $\lim_{n\to \infty} u_n(t,x)$ exists and
 $u(t,x)=\lim_{n\to\infty}u_n(t,x)$ is an entire solution of  \eqref{aux-new-eq01}.

Now, we claim that
\begin{equation}
\label{u-bdd-below}
u(t,x;-n\omega,\sigma\phi(0, \cdot;\bar c))\geq \sigma\min_{t\in[0,\omega]}\phi(t, x;\bar c) \quad \forall \,\, t> -n\omega,\,\, -l<x<l.
\end{equation}
Indeed, by  \eqref{aux-new-new-eq3-1} and the comparison principle, we have
\begin{equation}
\label{aux-new-new-eq4-0}
u(k\omega,x;-n\omega,\sigma\phi(0, \cdot;\bar c))>\sigma\phi(0, x;\bar c) \quad \forall \,\, k\ge -n+1,\,\, k\in \mathbb{Z},\,\, -l<x<l
\end{equation}
Suppose, by contraction, that \eqref{u-bdd-below} does not hold. Then there are $t_1>-n\omega$, $-l<x_1<l$ such that
\begin{equation}\label{contr-1}
u(t_1,x_1;-n\omega,\sigma\phi(0, \cdot;\bar c))<\sigma\min_{t\in[0,\omega]}\phi(t, x_{1};\bar c).
\end{equation}
In view of \eqref{aux-new-new-eq4-0}, without loss of generality, we may suppose that $t_1\in(-n\omega, (-n+1)\omega)$, and $t_1=-n\omega+\tilde\delta$ for $0<\tilde\delta<\omega$. By \eqref{aux-new-new-eq1}, it follows that
\begin{equation}
\label{aux-new-new-eq1-no}
u(-n\omega+\tilde\delta,x;-n\omega,\sigma \phi(0, \cdot;\bar c))\ge \sigma \tilde u(-n\omega+\tilde\delta,x;-n\omega,\phi(0, \cdot;\bar c),A)\,\,\,\forall\, -l\le x\le l.
\end{equation}
By \eqref{aux-new-new-eq2}, we obtain 
\begin{equation}
\label{aux-new-new-eq1-no-1}
\tilde u(-n\omega+\tilde\delta,x;-n\omega,\phi(0, \cdot;\bar c),A)\geq \min_{t\in[0,\omega]}\phi(t, x;\bar c)\,\,\,\forall\, -l\le x\le l.
\end{equation}
Combining \eqref{aux-new-new-eq1-no} and \eqref{aux-new-new-eq1-no-1}, we have
$$
u(-n\omega+\tilde\delta,x;-n\omega,\sigma \phi(0, \cdot;\bar c))\ge \sigma \min_{t\in[0,\omega]}\phi(t, x;\bar c)\,\,\,\forall\, -l\le x\le l,
$$
which contradicts \eqref{contr-1}. Thus, the claim holds and $u(t,x)$ is a positive bounded entire solution of \eqref{aux-new-eq01} satisfying \eqref{aux-new-new-eq0}.
\end{proof}

Now we are ready to prove Theorem \ref{compact-support}.

\begin{proof}[Proof of Theorem \ref{compact-support}]

(1) { Suppose $c>c^*=2\sqrt{r^*}$.}
Choose $\bar c$ and  $0<\kappa\le \sqrt{r^*}$ such that
$$\frac{(\kappa-\sqrt{\lambda})_{+}}{(\kappa +\sqrt{\lambda})}\leq \frac{2(b-\chi\mu)}{\chi\mu}$$
and
$$
c^*=2\sqrt {r^*}\leq c_\kappa<\bar c<c.
$$
By Lemma \ref{lem-001-1} (i),  there holds
$$
\lim_{t\to\infty}\sup_{|x|\ge \bar ct }  u(t,x;u_0)=0.
$$

We claim that $\lim_{t\to\infty}\sup_{x\in\R}  u(t,x;u_0)=0$. For otherwise, there are $\delta_0>0$, $t_n\to\infty$ and $x_n\in (-\bar c t_n, \bar c t_n)$
such that
$$
 u(t_n,x_n;u_0)\ge \delta_0\quad \forall\, n\ge 1.
$$
Let $u_n(t,x)=u(t+t_n,x+x_n;u_0)$ and $v_n(t,x)= v(t+t_n,x+x_n;u_0)$. Note  that
  $x_n-c t_n\to  -\infty$ as $n\to\infty$.
  Let $[\frac{t_n}{\omega}]$ be the greatest integer less than or equal to $\frac{t_n}{\omega}$. Then $t_{n}-[\frac{t_n}{\omega}]\omega\in [0, \omega)$.
  Without loss of generality, we may assume that 
  $\lim_{n\to\infty}(t_{n}-[\frac{t_n}{\omega}]\omega)=t^*$ and
  there is $(u^*(t,x),v^*(t,x))$ such that
$$
\lim_{n\to\infty}(u_n(t,x),v_n(t,x))=(u^*(t,x),v^*(t,x))
$$
locally uniformly in $(t,x)\in\R\times \R$, and $(u^*(t,x),v^*(t,x))$ satisfies
\begin{equation}\label{new-Keller-Segel-eq1}
\begin{cases}
u_t^*=u^*_{xx}-\chi ( u^*v^*_{x})_{x}+u^*(r(t+t^*, -\infty)-bu^*),\quad t\in\R,\,\, x\in\R\cr
0 =v^*_{xx}-  \lambda v^* +\mu u^*,\quad t\in\R,\,\, x\in\R.
\end{cases}
\end{equation}
By the periodicity of $r(t+t^*, \infty)$ in $t$ and $\int_{0}^{\omega}r(t+t^*,-\infty)dt<0$, it follows  that $u^*(t,x)\equiv 0$, which contradicts
$$u^*(0,0)=\lim_{n\to\infty}  u(t_n,x_n;u_0)\ge \delta_0.$$
Therefore, $\lim_{t\to\infty} \sup_{x\in\R}  u(t,x;u_0)=0$.

\medskip

(2) { Suppose $-c^*\leq c<c^*=2\sqrt{r^*}$.} We first prove
  $$
\lim_{t\to\infty}\sup_{x\le (c-\varepsilon)t}u(t,x;u_0)=0.
$$
Assume that the result does not hold. Then there are constants $\delta_0>0$, $\varepsilon_0>0$, and a sequence
$\{(t_n, x_n)\}_{n\in\N}$,  $t_n\to\infty$, $x_n\in (-\infty,(c-\varepsilon_0)t_n]$
such that 
$$
u(t_n,x_n;u_0)\ge \delta_0\quad \forall\, n\ge 1.
$$
Let $u_n(t,x)=u(t+t_n,x+x_n;u_0)$ and $v_n(t,x)=v(t+t_n,x+x_n;u_0)$.    Note  that
  $x_n-c t_n\to  -\infty$ as $n\to\infty$.
Let $[\frac{t_n}{\omega}]$ be the greatest integer less than or equal to $\frac{t_n}{\omega}$. Similarly, without loss of generality, we may assume that $\lim_{t\to\infty}(t_{n}-[\frac{t_n}{\omega}]\omega)=t^*$ and there is $(u^*(t,x), v^*(t,x))$ such that
$$
\lim_{n\to\infty}(u_n(t,x),v_n(t,x))=(u^*(t,x),v^*(t,x))
$$
locally uniformly in $(t,x)\in\R\times \R$, and $(u^*(t,x),v^*(t,x))$ satisfies \eqref{new-Keller-Segel-eq1}.
Again, by the periodicity of $r(t+t^*, -\infty)$ in $t$ and $\int_{0}^{\omega}r(t+t^*,-\infty)dt<0$, it can be proved that $u^*(t,x)\equiv 0$, which contradicts to
$$u^*(0,0)=\lim_{n\to\infty} u(t_n,x_n;u_0)\ge \delta_0.$$
Therefore, $\lim_{t\to\infty} \sup_{x\le (c-\varepsilon)t} u(t,x;u_0)=0$.

Next, we prove
$$
\lim_{t\to\infty}\sup_{x\ge (c^*+\varepsilon)t}u(t,x;u_0)=0.
$$
For any $\varepsilon>0$, let $\kappa=\sqrt{r^*}$, then $c_{\kappa}=2\sqrt{r^*}<c^*+\varepsilon$.
Since $\big(1+\frac{1}{2}\frac{(\sqrt{r^*}-\sqrt{\lambda})_+}{(\sqrt{r^*}+\sqrt{\lambda})}\big)\chi\mu \leq b$ and $0<\chi\mu<b$, Lemma \ref{lem-001-1} (i)
gives rise to
$$
u(t,x;u_0)\le Me^{-\kappa(|x|-c_{\kappa}t)}\max_{t\in[0,\omega]}\psi(t), \quad\forall\ x\in\R,\ t\geq 0,
$$
and
$$
\sup_{x\ge (c^*+\varepsilon)t} u(t,x;u_0)\le Me^{-\kappa(c^*+\varepsilon-c_\kappa)t}\max_{t\in[0,\omega]}\psi(t)\to0 \quad {\rm as} \quad  t\to\infty.
$$
Therefore, $\lim_{t\to\infty}\sup_{x\ge (c^*+\varepsilon)t}u(t,x;u_0)=0$.

\smallskip

We now prove
\begin{equation}\label{Tm1.1.2liminf}
\liminf_{t\to\infty}\inf_{(c+\varepsilon)t\le x\le (c^*-\varepsilon)t} u(t,x;u_0)>0.
\end{equation}

To this end, for any $0<\varepsilon<\frac{c^*-c}{2}$, let $l$ be as in \eqref{l-eq}.
Fix a $\bar c$ satisfying ${ -c^*+\varepsilon\leq}  c+\varepsilon\le \bar c\le c^*-\varepsilon$  and set $M=\max\{\|u_0\|_{\infty},\frac{\bar r}{b-\chi\mu}\}$. {Let $\lambda(\bar c)$ be as in \eqref{lambda-eq}.} By \eqref{E2-1},  for any  $R\gg 1, { p>1} $, $(\tilde u(t,x;u_0),\tilde v(t,x;u_0))$ satisfies
\begin{equation}\label{AA-eq1}
\tilde u_t\geq \tilde u_{xx}+\bar c \tilde u_x-\chi \tilde v_x\tilde u_x+\tilde u(r(t, x-(c-\bar c)t)-\varepsilon_R M- C_{R,p} \tilde u^{\frac{1}{p}}-(b-\chi\mu)\tilde u), \quad t\geq 1,\ x\in\R.
\end{equation}
Let $p=2$ and $\eta$ be as in
Lemma \ref{lem-001-3}. Choose  $R\gg 1$  such that $
\varepsilon_{R}M<\frac{\eta}{4}$ and
\begin{equation}\label{AA-eq2-1}
|\chi \tilde v_x|\leq C_{R}\sqrt {\tilde u(t,x;u_0)}+\frac{\eta}{4}, \quad t\geq 1,\ x\in\R.
\end{equation}
Define
$$
A(t,x)=\begin{cases}
\frac{\chi \tilde v_x(1,x;u_0)}{\max\{1, |\chi \tilde v_x(1,x;u_0)|\eta^{-1}\}}, \,\, &{\rm if} \,\, t< 1, \ x\in\R \cr
\frac{\chi \tilde v_x(t,x;u_0)}{\max\{1, |\chi \tilde v_x(t,x;u_0)|\eta^{-1}\}},     \,\, &{\rm if}\,\, t\geq 1,\ x\in\R.
\end{cases}
$$
 Note that $\tilde v_x(t,x;u_0)$ is globally H\"older continuous in $t\geq 1$ and $x\in \R$.  We then see that $A(t,x)$ is globally H\"older continuous in $t\in\R$ and
$x\in \R$. It is clear that  $\|A(\cdot,\cdot)\|_\infty<\eta$.

 Let $T\gg 1$ be such that
$r(t, x-(c-\bar c)t)\ge r(t, \infty)-\varepsilon_R M$ for $x\ge -l$ and $t\ge T$.
Choose $\bar b>b$ and also $\bar b \gg 1$ such that
 \begin{equation}\label{AA-eq2-2}
 C_{R}\sqrt{\frac{\bar r}{\bar b-\chi\mu}}+\frac{\bar r}{\bar b-\chi\mu}<\frac{\eta}{4},
 \end{equation}
 and
 $u(T,x;\bar c,\bar b, A)<\frac{\bar r}{\bar b-\chi\mu}< \tilde u(T,x;u_0)$  for $-l\le x\le l$ and $\bar c\in [c+\varepsilon,c^*-\varepsilon]$, where
$u(t,x;\bar c,\bar b, A)$ is a positive bounded entire solution of \eqref{aux-new-eq01}.

 \medskip

 Fix such a $\bar b$. We first claim that for any $c+\varepsilon \le \bar c\le c^*-\varepsilon$,
\begin{equation}\label{AA-eq3}
\tilde u(t,x;u_0)\ge u(t,x;\bar c,\bar b, A)\quad \forall\,\, t\ge T,\,\, -l\le x\le l.
\end{equation}

Suppose, by contradiction, that \eqref{AA-eq3} does not hold. Then there are $c+\varepsilon \le \bar c\le c^*-\varepsilon$ and $t_{\rm inf}\in[T,\infty)$ satisfying
$$
t_{\rm inf}:=\inf\{t\in(T,\infty)\, |\,  \exists  x_t\in\R \, \,  {\rm such \, \, that }\, \, u(t, x_t; \bar c,\bar b, A)\geq\tilde u(t,x_t; u_0), |x_t|\leq l\}.
$$
Since $u(T,x;\bar c,\bar b, A)<\tilde u(T,x;u_0)$ for any $-l\le x\le l$,
it follows that $t_{\rm inf}>T$.
Note that $u(t,-l;\bar c,\bar b, A)=u(t,l;\bar c,\bar b, A)=0$ for any $t\in\R$. Then there is $x_{\rm inf}\in\R$ such that $|x_{\rm inf}|<l$,
\begin{equation}\label{AA-eq7}
\frac{\bar r}{\bar b-\chi\mu}>u(t_{\rm inf},x_{\rm inf};\bar c,\bar b, A)=\tilde u(t_{\rm inf},x_{\rm inf};u_0),
\end{equation}
and
$$
u(t,x;\bar c,\bar b, A)<\tilde u(t,x;u_0),\quad\ |x|\leq l,\quad T\le t< t_{\rm inf}.
$$
 Hence there is $0<\delta\ll 1$ such that  $[t_{\rm inf}-\delta,t_{\rm inf}]\times [x_{\rm inf}-\delta,x_{\rm inf}+\delta]\subset \{(t,y) : T\le t\leq t_{\rm inf}, |y|<l\}$ and
$$
A(t,x)=\chi \tilde v_x(t,x;u_0)
,\quad \forall\ t_{\rm inf}-\delta\le t\leq t_{\rm inf},\quad x_{\rm inf}-\delta\leq x\leq x_{\rm inf}+\delta.$$
Note that
$$
u(t_{\rm inf}-\delta,x;\bar c,\bar b, A)<\tilde u(t_{\rm inf}-\delta,x;u_0) \quad \forall\,\, x\in[x_{\rm inf}-\delta,x_{\rm inf}+\delta]
$$
and
$$
u(t,x_{\rm inf}\pm\delta;\bar c,\bar b, A)\leq \tilde u(t,x_{\rm inf}\pm\delta; u_0) \quad \forall\,\, t_{\rm inf}-\delta\leq t\leq t_{\rm inf}.
$$
Thus, by the comparison principle, we have
$$
u(t,x;\bar c,\bar b, A)<\tilde u(t,x; u_0) \quad\forall\,\, t_{\rm inf}-\delta< t\leq t_{\rm inf}, \,\, x_{\rm inf}-\delta< x< x_{\rm inf}+\delta.
$$
In particular,
$$
u(t_{\rm inf},x_{\rm inf};\bar c,\bar b, A)<\tilde u(t_{\rm inf},x_{\rm inf};u_0).
$$
Which contradicts \eqref{AA-eq7}.

\medskip

By Lemma \ref{lem-001-3}, we have
$$
\inf_{-l+\delta\le x\le l-\delta,c+\varepsilon\le \bar c\le c^*-\varepsilon, t\ge T}u(t,x;\bar c,\bar b, A)>0, \quad \forall\,\, 0<\delta <l.
$$
This, together with \eqref{AA-eq3}, implies that
$$
\inf_{-l+\delta\le x\le l-\delta,c+\varepsilon\le \bar c\leq c^*-\varepsilon, t\ge T}\tilde  u(t,x;u_0)>0,  \quad \forall \,\, 0<\delta <l,$$
and hence,
$$
\inf_{-l+\bar ct+\delta\le x\le l+\bar ct-\delta, c+\varepsilon\leq \bar c\leq c^*-\varepsilon, t\ge T} u(t,x;u_0)>0,  \quad \forall \,\, 0<\delta <l.$$
It then follows that
$$
\inf_{-l+(c+\varepsilon)t+\delta\le x\le l+(c^*-\varepsilon)t-\delta, t\ge T} u(t,x;u_0)>0,  \quad \forall \,\, 0<\delta <l.$$
Thus,
$$
\liminf_{t\to\infty}\inf_{(c+\varepsilon)t\le x\le (c^*-\varepsilon)t} u(t,x;u_0)>0.
$$

Finally,  suppose that $2\chi\mu<b$. We prove
\begin{equation}\label{tm1.1-2}
\lim_{t\to\infty} \sup_{(c+\varepsilon)t\le x\le (c^*-\varepsilon)t}|u(t,x;u_0)-p(t)|=0.
\end{equation}
Suppose  by contraction that the result does not hold. Then there are constants $0< \varepsilon< \frac{c^*-c}{2}$, $\delta>0$, and a sequence $\{(x_n,t_n)\}_{n \in \N}$ such that $t_n \to \infty$, $t_n(c+\varepsilon)\leq x_n\leq t_n(c^*-\varepsilon)$, and
\begin{equation}\label{eq000}
|u(t_n,x_n; u_0)-p(t_{n})|\geq \delta \quad  \forall\, n\geq 1.
\end{equation}
For every $n\geq 1$, define $(u_n(t,x),v_n(t,x))=(u(t+t_n,x+x_n;u_0),v(t+t_n,x+x_n;u_0))$. 
Let $[\frac{t_n}{\omega}]$ be the greatest integer less than or equal to $\frac{t_n}{\omega}$. 
By a priori  estimates for parabolic equations, without loss of generality,  we may suppose that $(u_n(t,x),v_n(t,x))\to (u^*(t,x),v^*(t,x))$ locally uniformly in $C^{1,2}(\R\times\R)$ and $t_{n}-[\frac{t_n}{\omega}]\omega\to t^*$ as $n\to\infty$. Furthermore, $(u^*(t,x)$, $v^*(t,x))$  is an entire solution of the following system
\begin{equation}\label{new-Keller-Segel-eq1-infty}
\begin{cases}
u_t^*=u^*_{xx}-\chi ( u^*v^*_{x})_{x}+u^*(r(t+t^*, \infty)-bu^*),\quad t\in\R,\,\, x\in\R\cr
0 =v^*_{xx}-  \lambda v^* +\mu u^*,\quad t\in\R,\,\, x\in\R.
\end{cases}
\end{equation}

Choose $0<\tilde\varepsilon<\varepsilon<\frac{c^*-c}{2} $. For every $x\in \R$ and $t\in \R$, we have
\begin{align}
x+x_n&\leq x+t_n(c^*-\varepsilon) \cr
&=(c^*-\tilde\varepsilon)(t_n+t)-(\varepsilon-\tilde\varepsilon)(t_n-\frac{x-(c^*-\tilde\varepsilon)t}{\varepsilon-\tilde\varepsilon})\cr
&\leq (t+t_n)(c^*-\tilde\varepsilon)
\end{align}
whenever $t_n\geq \frac{\|x\|+(c^*-\tilde\varepsilon)|t|}{\varepsilon-\tilde\varepsilon}$.
On the other hand, For every $x\in \R$ and $t\in \R$, we have
\begin{align}
x+x_n&\geq x+t_n(c+\varepsilon) \cr
&=(c+\tilde\varepsilon)(t_n+t)+(\varepsilon-\tilde\varepsilon)(t_n-\frac{(c+\tilde\varepsilon)t-x}{\varepsilon-\tilde\varepsilon})\cr
&\geq (t+t_n)(c+\tilde\varepsilon)
\end{align}
whenever $t_n\geq \frac{\|x\|+(c+\tilde\varepsilon)|t|}{\varepsilon-\tilde\varepsilon}$.
It then follows that
$$(t+t_n)(c+\tilde\varepsilon)\leq x+x_n\leq (t+t_n)(c^*-\tilde\varepsilon)$$
whenever $t_n\geq \max\{\frac{\|x\|+(c^*-\tilde\varepsilon)|t|}{\varepsilon-\tilde\varepsilon}, \frac{\|x\|+(c+\tilde\varepsilon)|t|}{\varepsilon-\tilde\varepsilon} \}$.
Since
$$
u^*(t,x)=\lim_{n \to \infty}u(t+t_n,x+x_n;u_0)\geq \liminf_{s \to \infty}\inf_{s(c+\tilde\varepsilon)\leq y\leq s(c^*-\tilde\varepsilon)}u(s,y;u_0) >0
$$
for every $(t,x)\in \R \times \R$, we obtain $\inf_{(t,x)\in \R \times \R } u^*(t,x)>0$.
By  \cite[Theorem 1.5]{SaSh4},
we must have $u^*(t,x)=p(t+t^*)$ for every $(t,x)\in \R\times\R$. In particular, $\lim_{n\to\infty}u(t_n, x_n; u_0)=u^*(0,0)=p(t^*)$. 
In view of $\lim_{t\to \infty}p(t_n)=\lim_{t\to \infty}p(t_n-[\frac{t_n}{\omega}]\omega)=p(t^*)$, it follows that $\lim_{n\to\infty}(u(t_n, x_n; u_0)-p(t_n))=0$,
which contradicts \eqref{eq000}.

\medskip

 (3) First, let $\kappa=\sqrt{r^*}$, and $\bar c$ be such that
$$
c^*=2\sqrt {r^*} =c_\kappa<\bar c<|c|.
$$
By Lemma \ref{lem-001-1} (i), it follows that
$$
\lim_{t\to\infty}\sup_{|x|\ge \bar ct }  u(t,x;u_0)=0.
$$
This implies that for any $\varepsilon>0$,
$$
\lim_{t\to\infty}\sup_{|x|\ge (c^*+\varepsilon)t }  u(t,x;u_0)=0.
$$

Next,  for any $0<\varepsilon<c^*$,  let $l$ be as in \eqref{l-eq}.
Fix a $\bar c$ satisfying ${ -c^*+\varepsilon\leq}  \bar c\le c^*-\varepsilon$. Let $\lambda(\bar c)$ be as in \eqref{lambda-eq}.
By the  arguments similar to  those in the proof of \eqref{Tm1.1.2liminf}, it follows that
$$
\inf_{-l+\bar ct+\delta\le x\le l+\bar ct-\delta, -c^*+\varepsilon\leq \bar c\leq c^*-\varepsilon, t\geq T} u(t,x;u_0)>0,  \quad \forall \,\,  0<\delta <l.
$$
Thus, we have 
$$
\inf_{-l+(-c^*+\varepsilon)t+\delta\le x\le l+(c^*-\varepsilon)t-\delta, t\ge T} u(t,x;u_0)>0, \quad \forall \,\, 0<\delta <l$$
and hence, 
$$
\liminf_{t\to\infty}\inf_{(-c^*+\varepsilon)t\le x\le (c^*-\varepsilon)t} u(t,x;u_0)>0.
$$
Moreover,  using the similar arguments as those in the proof of \eqref{tm1.1-2}, we can prove that
if $2\chi\mu<b$, then
$\lim_{t\to\infty} \sup_{(-c^*+\varepsilon)t\le x\le (c^*-\varepsilon)t}|u(t,x;u_0)-p(t)|=0$.
\end{proof}

Next we prove  Theorem \ref{right-support}.

\begin{proof}[Proof of Theorem \ref{right-support}]
(1)   { Suppose $c\geq-c^*=-2\sqrt{r^*}$.} By the same arguments as those in Theorem \ref{compact-support} (2),
it follows  that
$$\lim_{t\to\infty}\sup_{x\leq(c-\varepsilon)t}u(t,x;u_0)=0.$$

Next, we prove that
\begin{equation}\label{Tm1.2-1liminf-c}
\liminf_{t\to\infty}\inf_{x\ge  (c+\varepsilon)t}u(t,x;u_0)>0\quad \forall\,\,  \varepsilon>0.
\end{equation}

To this end, for any $\tilde \varepsilon>0 $, let $u(t,x;u_0)=\tilde u(t,x-(c+\tilde \varepsilon)t;u_0), v(t,x;u_0)=\tilde v(t,x-(c+\tilde \varepsilon)t;u_0) $ in \eqref{Keller-Segel-eq0}  and set $M=\max\{\|u_0\|_{\infty},\frac{\bar r}{b-\chi\mu}\}$. By \eqref{E2-1}, it follows that, for any  $R\gg 1,{ p>1}$, $(\tilde u(t,x;u_0),\tilde v(t,x;u_0))$ satisfies
\begin{equation}\label{AA-eq1-1-1}
\tilde u_t\geq \tilde u_{xx}+(c+\tilde \varepsilon) \tilde u_x-\chi \tilde v_x\tilde u_x+\tilde u(r(t, x+\tilde \varepsilon t)-\varepsilon_R M- C_{R,p} \tilde u^{\frac{1}{p}}-(b-\chi\mu)\tilde u), \quad t\geq 1,\ x\in\R.
\end{equation}
Let $p=2$. Choose  $R\gg 1$, $0<\xi\ll 1$ and $0<\epsilon\ll \min\{1, 2\sqrt{r^*}\}$
such that $
\varepsilon_{R}M<\frac{\xi}{4}$,
\begin{equation}\label{AA-eq2-1}
|\chi \tilde v_x|\leq C_{R}\sqrt {\tilde u(t,x;u_0)}+\frac{\xi}{4}, \quad t\geq 1,\ x\in\R,
\end{equation}
and
\begin{equation}\label{AA-eq2-New1}
-c^*+\epsilon=-2\sqrt{r^*}+\epsilon \leq c +\tilde \varepsilon -\xi.
\end{equation}
Define
$$
B(t,x)=\frac{\chi \tilde v_x(t, x; u_0)}{\max\{1, |\chi \tilde v_x(t, x; u_0)|\xi^{-1}\}}, \quad t\geq 1,\ x\in\R.
$$
For the remaining part of the proof, we proceed with four steps.

\smallskip

\noindent{\bf Step 1.} {\it  We construct some sub-solution for \eqref{AA-eq1-1-1}.}

\smallskip

First, choose $0<\xi_1\ll 1$  such that
 \begin{equation}\label{AA-eq2-2}
 C_R\sqrt{\xi_1}+\xi_1<\frac{\xi}{4}.
 \end{equation}
Next, let $l$ be chosen as in \eqref{l-eq}.  Let $T>1$ be such that $r(t, x+\tilde \varepsilon t)\ge r(t, \infty)-\varepsilon_R M$ for $x\ge -l$ and $t\ge T$, and 
 let $\underline u_1(\cdot)\in C^b_{\rm unif}([-l,\infty))\setminus\{0\}$ be such that
$$
\underline u_1(-l)=0,\quad  \underline u'_1(x)>0, \quad {\rm and} \quad \underline u_1(x)<\tilde u(T,x;u_0), \quad \forall \,\, x\geq -l.
$$
 Choose $ -2\sqrt{r^*}+\epsilon \leq\bar c\leq\min\{2\sqrt{r^*}-\epsilon,(c+\tilde \varepsilon)-\xi\}$. Let $\underline u(t,x)$ be the solution of
\begin{equation}
\label{aux-new-eq011}
\begin{cases}
\underline u_t=\underline u_{xx}+\bar c\underline u_x+\underline u (r(t, \infty) -2\varepsilon_R M- C_{R,p} \underline u^{\frac{1}{p}}-{ (\frac{\bar r}{\xi_1}+(b-\chi\mu) ) } \underline u), \quad t >T,\ x>-l \cr
\underline u(t,-l)=0 \cr
\underline u(T,x)=\frac{\xi_1}{M+\xi_1}\underline u_1(x).
\end{cases}
\end{equation}

Note that $\underline u(t,x)\equiv \xi_1$ is a super-solution of \eqref{aux-new-eq011} and $\|\underline{u}(T,\cdot)\|_{\infty}<\xi_1$. 
Then the comparison principle implies 
$$
\underline{u}(t,x)<\xi_1, \quad\forall\ t\geq T, \ x\geq-l.
$$
Since $\underline u'_1(x)>0$, we have $\underline u_{x}(t,x)>0$ {for any $t\geq T$, $x\geq -l$.}
Note that $|B(t,x)|<\xi$ for all $t\geq 1$, $x\in\R$. 
Thus, $\underline u(t,x)$ satisfies
\begin{align*}
\underline{u}_t= &\underline u_{xx}+\bar c\underline u_x+\underline u (r(t, \infty) -2\varepsilon_R M- C_{R,p} \underline u^{\frac{1}{p}}-{ (\frac{\bar r}{\xi_1}+(b-\chi\mu) ) } \underline u)\cr
<&\underline u_{xx}+(c+\tilde \varepsilon-B(t,x))\underline u_x+\underline u (r(t, \infty) -2\varepsilon_R M- C_{R,p} \underline u^{\frac{1}{p}}-(b-\chi\mu)\underline u) \quad\forall\ t>T, \ x>-l.
\end{align*}

\smallskip

\noindent {\bf Step 2.} {\it We show that}
\begin{equation}\label{AA-Stp2eq1}
\liminf_{t\to\infty}\inf_{x\ge -l+\delta} \underline u(t,x)>0,{\quad  \forall \,\, \delta>0.  }
\end{equation}

\smallskip

{
Choose $\bar b>\frac{\bar r}{\xi_1}+b$ and also }$\bar b\gg 1$ such that $u(T{ +1},x;\bar c,\bar b)\leq\underline u(T { +1} ,x)$  for $-l\le x\le l$, where $u(t,x;\bar c,\bar b)$ is a positive entire solution of \eqref {aux-new-eq01} with $A(t,x)=0$. Fix such $\bar b$.
 It follows from the comparison principle that
 $$
 u(t,x;\bar c,\bar b)<\underline u(t,x), \quad t>T{ +1} ,\ -l<x<l.
 $$
 Repeating the same procedure, by induction, we get
 $$
 u(t,x-kl;\bar c,\bar b)<\underline u(t,x), \quad t>T{ +1}  ,\ (k-1)l<x<(k+1)l, \ k=0,1,2,\cdots
  $$
 This implies that there exists $\delta_0 >0$ such that
 $$\inf_{(k-1)l+\delta \leq x\leq (k+1)l-\delta, t\in\R}u(t,x-kl;\bar c,\bar b)>\delta_0,$$
 for any $0<\delta<l, \ k=0,1,2,\cdots$. Thus, we have
 $$
 \liminf_{t\to\infty}\inf_{x\geq -l+\delta}\underline u(t,x)>0,{\quad  \forall \,\, \delta>0.  }
 $$

 \noindent {\bf Step 3.} {\it  We show that}
\begin{equation}\label{AA-Ste3eq1}
\underline{u}(t,x)\leq \tilde u(t,x;u_0), \quad \forall\ x\geq -l , \ t\ge T.
\end{equation}

Note that
$$
\underline{u}(t,-l)=0<\tilde u(t,-l;u_0)\quad \forall\, \, t\ge T,
$$
$$
\underline{u}(T,x)<\tilde u(T,x;u_0)\quad \forall\, x\ge -l,
$$
and
$$
\tilde u_\infty(T):=\liminf_{x\to \infty}\tilde u(T,x;u_0)>\underline{u}_\infty(T):=\lim_{x\to\infty}\underline{u}(T,x).
$$
For given $t\ge T$, let
$$
\tilde u_\infty(t)=\liminf_{x\to \infty}\tilde u(t,x;u_0),\quad \underline{u}_\infty(t)=\lim_{x\to\infty}\underline{u}(t,x).
$$
We claim that
$$
\tilde u_\infty (t)>\underline{u}_\infty(t)\quad \forall\,\,\ t>T.
$$
In fact, for any given $t_0>T$, there is $x_n\to\infty$ such that
$$
\tilde u(t_0,x_n;u_0)\to \tilde u_\infty(t_0),\quad \underline{u}(t_0,x_n)\to \underline{u}_\infty(t_0)
$$
as $n\to\infty$.
Without loss of generality, we may assume that
$$
\tilde u(t,x+x_n;u_0)\to  \tilde u^*(t,x),\quad \underline{u}(t,x+x_n)\to \underline{u}^*(t)
$$
as $n\to\infty$ locally uniformly in $(t,x)\in (T,\infty)\times \R$.
By $\tilde u_\infty(T)>\underline{u}_\infty(T)$ and the comparison principle, it follows that
$$
\tilde u^*(t,x)>\underline{u}^*(t)\quad \forall\,\, t>T,\quad x\in\R.
$$
In particular, we have
$$
\tilde u_\infty(t_0)=\tilde u^*(t_0,0)>\underline{u}_\infty(t_0).
$$
This shows that  the claim holds true.

Next, assume that there are $t_{1}>T$ and $x_{1}>-l$ such that $\tilde u(t_{1},x_{1};u_0)<\underline{u}(t_{1},x_{1})$.
Then there is $t_{\inf}>T$ such that
$$
\tilde u(t,x;u_0)>\underline{u}(t,x)\quad \forall\,\, T\le t<t_{\inf},\,\, x\ge -l.
$$
%and
%$$
%\inf_{x\ge -l}\big(\tilde u(t_{\inf},x;u_0)- \underline u(t_{\inf},x)\big)=0.
%$$
By the above claim and the definition of $t_{\inf}$, there is $x_{\inf}\in (-l,\infty)$ such that
$$
\tilde u(t_{\inf},x_{\inf};u_0)=\underline{u}(t_{\inf},x_{\inf}).
$$
Thus,  the arguments similar to those in the proof of \eqref{AA-eq3}
implies 
$$
\tilde u(t_{\inf},x_{\inf};u_0)>\underline u(t_{\inf},x_{\inf}),
$$
which is a contradiction. This shows that  \eqref{AA-Ste3eq1} holds
true.

\smallskip

\noindent {\bf Step 4.} {\it  We prove \eqref{Tm1.2-1liminf-c}.}

By \eqref{AA-Stp2eq1} and \eqref{AA-Ste3eq1}, we deduce that $$
\liminf_{t\to\infty}\inf_{x\ge -l+\delta} \tilde u(t,x;u_0)>0, {\quad  \forall \,\, \delta>0.  }
$$
Since $u(t,x;u_0)=\tilde u(t,x-(c+\tilde \varepsilon)t; u_0)$, we have
$$
\liminf_{t\to\infty}\inf_{x\ge -l+(c+\tilde \varepsilon)t+\delta} u(t,x;u_0)>0 \quad  \forall \,\, \delta>0.  
$$
Hence, for any $\varepsilon>0$
$$
\liminf_{t\to\infty}\inf_{x\ge (c+\varepsilon)t} u(t,x;u_0)>0.
$$

\smallskip

Finally, we prove that if $b>2\chi\mu$,
$$
\lim_{t\to\infty} \sup_{x\geq (c+\varepsilon)t}|u(t,x;u_0)-p(t)|=0.
$$
 This can be proved by the  arguments similar to those in the proof of \eqref{tm1.1-2}.

 \medskip

 (2) Suppose $c<-c^*=-2\sqrt{r^*}$.
First, let $\kappa=\sqrt{r^*}$, and let  $\bar c$ be such that
$$
c<\bar c<-c^*=-2\sqrt {r^*}=-c_{\kappa}.
$$
By Lemma \ref{lem-001-1} (ii), we have 
$$
\lim_{t\to\infty}\sup_{x\le \bar ct }  u(t,x;u_0)=0.
$$
This implies that for any $\varepsilon>0$,
\begin{equation}\label{u-left}
\lim_{t\to\infty}\sup_{x\le (-c^*-\varepsilon)t }  u(t,x;u_0)=0.
\end{equation}

Next, we prove for any $\varepsilon>0$,
\begin{equation}\label{right-support-1-1}
\liminf_{t\to\infty}\inf_{x\ge  (-c^*+\varepsilon)t}u(t,x;u_0)>0.
\end{equation}
It suffices to prove that for any $0<\varepsilon<2\sqrt{r^*}$, \eqref{right-support-1-1} holds.
Let $0<\varepsilon< 2\sqrt{r^*}$ be given, let $u(t,x;u_0)=\tilde u(t,x-(-c^*+\varepsilon)t;u_0), v(t,x;u_0)=\tilde v(t,x-(-c^*+\varepsilon)t;u_0) $ in \eqref{Keller-Segel-eq0}  and set $M=\max\{\|u_0\|_{\infty},\frac{\bar r}{b-\chi\mu}\}$. By \eqref{E2-1}, it follows that, for any  $R\gg 1,{ p>1}$, $(\tilde u(t,x;u_0),\tilde v(t,x;u_0))$ satisfies
\begin{equation}\label{3AA-eq1}
\tilde u_t\geq \tilde u_{xx}+(-c^*+\varepsilon) \tilde u_x-\chi \tilde v_x\tilde u_x+\tilde u(r(t, x+(-c^*+\varepsilon-c) t)-\varepsilon_R M- C_{R,p} \tilde u^{\frac{1}{p}}-(b-\chi\mu)\tilde u), \quad t\geq 1,\ x\in\R.
\end{equation}
Choose $0<\xi\ll \min\{1,\frac{\varepsilon}{2}\}$. Fix a $\bar c$ satisfying $-c^*+\frac{\varepsilon}{2}\leq \bar c\leq -c^*+\varepsilon-\xi$.
By the similar arguments as those in the proof of \eqref{Tm1.2-1liminf-c}, it can be proved that
$$
\liminf_{t\to\infty}\inf_{x\ge -l+\delta} \tilde u(t,x;u_0)>0, {\quad  \forall \,\, \delta>0.  }
$$
Since $u(t,x;u_0)=\tilde u(t,x-(-c^*+\varepsilon)t)$, we have
$$
\liminf_{t\to\infty}\inf_{x\ge -l+(-c^*+\varepsilon)t+\delta} u(t,x;u_0)>0, {\quad  \forall \,\, \delta>0.  }$$
Hence, for any $\varepsilon>0$,
$$
\liminf_{t\to\infty}\inf_{x\ge (-c^*+\varepsilon)t} u(t,x;u_0)>0.
$$

\smallskip

Finally, we prove if $b>2\chi\mu$,
$$
\lim_{t\to\infty} \sup_{x\geq (-c^*+\varepsilon)t}|u(t,x;u_0)-p(t)|=0.
$$
This can be proved by the arguments similar to those in the proof of \eqref{tm1.1-2}.
\end{proof}

\section{Forced wave solutions}

In this section, we study the existence and nonexistence of forced wave solutions of \eqref{Keller-Segel-eq0} and  prove Theorem \ref{forced-wave} and Theorem \ref{nonexistence-forced-wave}.  Throughout this section, we assume that {\bf (H1)} holds. We start with  a series of lemmas.

\smallskip

Suppose that $b>\chi\mu$. 
Let $x_1<0$ be given satisfying that $r(t, x_1)<0$ for all $t\in\R$, let $r_1=\sup_{t\in\R} r(t, x_1)<0$, %Note that $\sup_{t\in\R} r(t, x_1)=\max_{t\in [0,\omega]} r(t, x_1)<0$. Therefore, $r_1<0$.
and let $\theta$ be the positive root of the equation $\theta^2+c\theta+r_1=0$.
Define
\begin{equation}\label{U-Def}
U^{+}(x)=\min\{\frac{\bar r}{b-\chi\mu}, \frac{\bar r}{b-\chi\mu}e^{\theta(x-x_1)}\},
\end{equation}
and 
\begin{equation}\label{E-set}
\mathcal{E}^+:=\{u\in C(\R, C^b_{\rm unif}(\R)) : 0\leq u(t, x)\leq U^{+}(x),\,\ u(t, x)=u(t+\omega, x)\,\ \forall\,\ t\in\R,\, x\in\R\}.
\end{equation}
For every $u\in\mathcal{E}^+$, we consider the operator
\begin{equation}\label{A-equ}
\mathcal{A}_{u}(U)=\p_{xx}U+(c-\chi\p_{x}\Psi(\cdot,\cdot;u))\p_{x}U+(r(t, x)-\chi\lambda\Psi(\cdot,\cdot;u)-(b-\chi\mu)U)U-\p_{t}U,
\end{equation}
where $\Psi(t, x;u)$ is given by \eqref{psi-definition}.
 
A straightforward calculation, as shown in \cite[Lemma 3.1]{ShXu2}, 
gives rise to the  following result.

\begin{lem}
\label{super-solu-lm}
Suppose that {$b\geq \frac{3}{2}\chi\mu$}. For every $u\in \mathcal{E}^+$, it holds that $\mathcal{A}_{u}(\frac{\bar r}{b-\chi\mu})(t, x)\leq 0$ for $t\in\R$, $x\in\R$ and
$\mathcal{A}_{u}(\frac{\bar r}{b-\chi\mu}e^{\theta(\cdot-x_1)})(t, x)\leq 0$ for $t\in\R$, $x\in (-\infty,x_1)$.
\end{lem}

%\begin{proof}
%It can be proved by the similar calculation in \cite[Lemma 3.1]{ShXu2}.
%	{\color{red} revise?,   write a sentence in front of this lemma}
%\end{proof}

Suppose $b>(1+\frac{\bar r}{\underline r})\chi\mu$. For any $0<\varepsilon\ll 1$, we  define an ignition nonlinearity by
\begin{equation}\label{ignition-type}
f_{\varepsilon}(t, u)=\begin{cases}
u\left(r(t, \infty)-\varepsilon -\frac{\chi\mu \bar r}{b-\chi\mu}-(b-\chi\mu)u\right), \,\, &{\rm if} \,\, u\geq 0,\cr
0                     \,\, &{\rm if} -\varepsilon\leq u<0, 
\end{cases}
\end{equation}
and consider the following reaction-diffusion equation
\begin{equation}\label{ignition-eq}
u_t=u_{xx}+f_{\varepsilon}(t, u),\quad x\in \R, \,\, t\in\R.
\end{equation}
{Equation \eqref{ignition-eq} has a spatially decreasing time periodic traveling wave solution $u^{\varepsilon}(t, x)$  (see \cite[Theorem 1.2]{SheShe}), that is, there are a constant $c_{\varepsilon}\in \R$ and a function $\phi^{\varepsilon}(\cdot,\cdot)$ such that $u^{\varepsilon}(t, x)=\phi^{\varepsilon}(t, x-c_{\varepsilon}t)$ for $x\in\R$ and $t\in\R$ satisfying
\begin{equation}\label{ignit-TV0-eq}
\begin{cases}
\phi^{\varepsilon}_{t}=\phi^{\varepsilon}_{xx}+c_{\varepsilon}\phi^{\varepsilon}_{x}+f_{\varepsilon}(t, \phi^{\varepsilon}),\,\ t\in\R, \,\, x\in\R,\cr
\lim_{x\to -\infty}\phi^{\varepsilon}(t, x)=q^{\varepsilon}(t), \,\, \lim_{x\to \infty}\phi^{\varepsilon}(t, x)=-\varepsilon \,\, {\rm uniformly}\,\, {\rm in}\,\ t\in\R,\cr
\phi^{\varepsilon}(t, x)=\phi^{\varepsilon}(t+\omega, x),\,\ \phi^{\varepsilon}_{x}(t,x)<0, \,\ t\in\R, \,\, x\in\R,
\end{cases}
\end{equation}
where $q^{\varepsilon}(t)$ is the unique positive $\omega$-periodic solution of 
\begin{equation}\label{q-periodic}
q^{\varepsilon}_{t}=q^{\varepsilon}(r(t, \infty)-\varepsilon -\frac{\chi\mu \bar r}{b-\chi\mu}-(b-\chi\mu)q^{\varepsilon}).
\end{equation}
}
Let $\psi^{\varepsilon}(t, x)=\phi^{\varepsilon}(t, -x)$ for any $t\in\R$, $x\in\R$. It then follows from \eqref{ignit-TV0-eq} that
\begin{equation}\label{ignit-TV}
\begin{cases}
\psi^{\varepsilon}_{t}=\psi^{\varepsilon}_{xx}-c_{\varepsilon}\psi^{\varepsilon}_{x}+f_{\varepsilon}(t, \psi^{\varepsilon}), \,\ t\in\R, \,\, x\in\R,\cr
\lim_{x\to\infty}\psi^{\varepsilon}(t, x)= q^{\varepsilon}(t), \,\, \lim_{x\to-\infty}\psi_{\varepsilon}(t, x)=-\varepsilon \,\, {\rm uniformly}\,\, {\rm in}\,\ t\in\R,\cr
\psi^{\varepsilon}(t, x)=\psi^{\varepsilon}(t+\omega, x), \,\, \psi^{\varepsilon}_{x}(t,x)>0, \,\ t\in\R, \,\, x\in\R.
\end{cases}
\end{equation}
For each $t\in\R$, there exists $x_0(t)\in\R$ such that $\psi^{\varepsilon}(t, x_{0}(t))=0$. Without loss of generality, we can assume that $r(t, x_0(t))\ge r(t, \infty)-\varepsilon$ for $t\in\R$ and $x_0(t)>x_1$ for $t\in\R$.
This can be achieved by some appropriate space translation of $\psi^{\varepsilon}(t, x)$ if necessary.

\begin{lem}\label{c-vare-limit}
$\lim_{\varepsilon\to 0^+} c_{\varepsilon}=2\sqrt{\frac{1}{\omega}\int_{0}^{\omega}(r(t, \infty)-\frac{\chi\mu\bar r}{b-\chi\mu})dt}$.
\end{lem}

\begin{proof}
First of all, we claim that $c_{\varepsilon}$ is nonincreasing in $\varepsilon>0$. Let $\varepsilon_2>\varepsilon_1>0$, and $u^{\varepsilon_1}(t,x)=\phi^{\varepsilon_1}(t, x-c_{\varepsilon_1}t)$ and $u^{\varepsilon_2}(t,x)=\phi^{\varepsilon_2}(t, x-c_{\varepsilon_2}t)$ be spatially decreasing time periodic traveling wave solution of \eqref{ignition-eq} with $\varepsilon$ being replaced by $\varepsilon_1$ and $\varepsilon_2$ respectively. By the definition of $f_{\varepsilon}$, $f_{\varepsilon_1}> f_{\varepsilon_2}$. Since any space translation of a traveling wave solution is still a traveling wave solution of \eqref{ignition-eq}, without loss of generality, we may assume that $u^{\varepsilon_1}(0, x)>u^{\varepsilon_2}(0, x)$ for any $x\in\R$ and $u^{\varepsilon_2}(0, 0)>-\frac{\varepsilon_1}{2}$. By the comparison principle, we have 
$$
u^{\varepsilon_1}(t,x)=\phi^{\varepsilon_1}(t, x-c_{\varepsilon_1}t)>u^{\varepsilon_2}(t,x)=\phi^{\varepsilon_2}(t, x-c_{\varepsilon_2}t),\quad \forall \,\ t>0, \,\ x\in\R.
$$
We further claim that $c_{\varepsilon_1}\geq c_{\varepsilon_2}$. Indeed, if not, $c_{\varepsilon_1} <c_{\varepsilon_2}$. Since $\phi^{\varepsilon_i}(t, z)$ is $\omega$-periodic in $t$, $i=1,2$, for every positive integer $k$, we obtain 
$$
\phi^{\varepsilon_1}(0, x-c_{\varepsilon_1}k\omega)=\phi^{\varepsilon_1}(k\omega, x-c_{\varepsilon_1}k\omega)>\phi^{\varepsilon_2}(k\omega, x-c_{\varepsilon_2}k\omega)=\phi^{\varepsilon_2}(0, x-c_{\varepsilon_2}k\omega).
$$
Taking $x=c_{\varepsilon_2}k\omega$ in above inequality, we get $\phi^{\varepsilon_2}(0,0)\leq \phi^{\varepsilon_1}(0, (c_{\varepsilon_2}-c_{\varepsilon_1})k\omega)\to -\varepsilon_1$ as $k\to\infty$, which contradicts to the assumption $\phi^{\varepsilon_2}(0,0)=u^{\varepsilon_2}(0, 0)>-\varepsilon_1$. Hence, we have $c_{\varepsilon_1}\geq c_{\varepsilon_2}$, which implies that $c_{\varepsilon}$ is nonincreasing in $\varepsilon>0$. It is easy to see that for any small $\varepsilon>0$, $c_{\varepsilon}\leq 2\sqrt{\frac{1}{\omega}\int_{0}^{\omega}(r(t, \infty)-\frac{\chi\mu\bar r}{b-\chi\mu})dt}$. Thus, $\lim_{\varepsilon\to 0^{+}}c_{\varepsilon}$ exists. Let $\tilde c=\lim_{\varepsilon\to 0^{+}}c_{\varepsilon}$. We then have $\tilde c\leq 2\sqrt{\frac{1}{\omega}\int_{0}^{\omega}(r(t, \infty)-\frac{\chi\mu\bar r}{b-\chi\mu})dt}$.

Consider a sequence $\{\varepsilon_{n}\}$ such that $\varepsilon_{n}\to 0$ as $n\to\infty$. For each $\varepsilon_{n}$, $n=1, 2, \cdot \cdot \cdot$, there exists a solution $(c_{\varepsilon_{n}}, \phi^{\varepsilon_n})$ satisfying \eqref{ignit-TV0-eq} with $\varepsilon$ being replaced by $\varepsilon_{n}$. Let $q(t)$ be the unique positive $\omega$-periodic solution of the following equation
\begin{equation}\label{no-epsi-eq}
q_{t}=q(r(t, \infty)-\frac{\chi\mu \bar r}{b-\chi\mu}-(b-\chi\mu)q). 
\end{equation}
Without loss of generality, we may assume that $\phi^{\varepsilon_n}(0, 0)=\frac{1}{2}\min_{t\in[0, \omega]}q(t)$ for any $n=1, 2, \cdot \cdot \cdot$, $c_{\varepsilon_{n}}\to \tilde c$ as $n\to\infty$ and $\phi^{\varepsilon_{n}}(t,x)\to \phi(t,x)$ as $n\to\infty$ locally uniformly in $t\in\R$ and $x\in\R$. Moreover, $(\tilde c, \phi)$ satisfies
\begin{equation*}\label{ignit-TV0}
\begin{cases}
\phi_{t}=\phi_{xx}+\tilde c\phi_{x}+\phi(r(t, \infty)-\frac{\chi\mu \bar r}{b-\chi\mu}-(b-\chi\mu)\phi),\,\ t\in\R, \,\, x\in\R,\cr
\phi(0, 0)=\frac{1}{2}\min_{t\in[0, \omega]}q(t)\cr
\phi(t, x)=\phi(t+\omega, x), \,\ t\in\R, \,\, x\in\R,
\end{cases}
\end{equation*}
Since $\phi^{\varepsilon_{n}}_{x}(t,x)<0$ and $0\leq \phi^{\varepsilon_{n}}(t,x)\leq \max_{t\in[0, \omega]}q(t)$, we have $\phi_{x}(t,x)\leq 0$ and $0\leq \phi(t,x)\leq \max_{t\in[0, \omega]}q(t)$, which implies that $\phi(t, \pm\infty):=\lim_{x\to\pm\infty}\phi(t,x)$ exist for any $t\in\R$ and are periodic solution of \eqref{no-epsi-eq}. By the uniqueness of positive $\omega$-periodic solution of \eqref{no-epsi-eq}, we have $\phi(t, \pm\infty)\in\{0, q(t)\}$.  In view of the monotonicity of $\phi(t,x)$ in $x$, we obtain
$$
\phi(0, +\infty)\leq \phi(0,0)=\frac{1}{2}\min_{t\in[0, \omega]}q(t)\leq \phi(0, -\infty),
$$
which implies that $\phi(t, +\infty)=0$ and $\phi(t, -\infty)=q(t)$. Therefore, $w(t,x)=\phi(t, x-\tilde ct)$ is a periodic traveling wave solution of 
\begin{equation}\label{w-periodic}
w_{t}=w_{xx}+w(r(t, \infty)-\frac{\chi\mu \bar r}{b-\chi\mu}-(b-\chi\mu)w)
\end{equation}
connecting $q(t)$ and $0$. Since $2\sqrt{\frac{1}{\omega}\int_{0}^{\omega}(r(t, \infty)-\frac{\chi\mu\bar r}{b-\chi\mu})dt}$ is the minimal speed of time periodic positive traveling wave solutions of \eqref{w-periodic} connecting $q(t)$ and $0$, it follows
that  
$$\tilde c\geq 2\sqrt{\frac{1}{\omega}\int_{0}^{\omega}(r(t, \infty)-\frac{\chi\mu\bar r}{b-\chi\mu})dt}. $$
Therefore, $\tilde c=2\sqrt{\frac{1}{\omega}\int_{0}^{\omega}(r(t, \infty)-\frac{\chi\mu\bar r}{b-\chi\mu})dt}$. 
 \end{proof}

\begin{lem}\label{subsolution}
For every  $u\in \mathcal{E}^+$ and $0<\varepsilon\ll 1$, $U^{-}(t, x)=\max\{\psi^{\varepsilon}(t, x),0\}$ satisfies that
$\mathcal{A}_u(U^{-}(\cdot, \cdot))(t, x)\geq 0$ for any $x\not=x_0(t)$.
Moreover, $U^{-}(t, x)<U^{+}(x)$ for all $t\in\R$, $x\in\R$.
\end{lem}

\begin{proof}
Let $\delta =c-\frac{\chi\mu \bar r }{2\sqrt{\lambda}(b-\chi\mu)}+2\sqrt{\frac{1}{\omega}\int_{0}^{\omega}(r(t, \infty)-\frac{\chi\mu\bar r}{b-\chi\mu})dt}>0$.
By Lemma \ref{c-vare-limit}, for $0<\varepsilon\ll 1$, $c_{\varepsilon}>2\sqrt{\frac{1}{\omega}\int_{0}^{\omega}(r(t, \infty)-\frac{\chi\mu\bar r}{b-\chi\mu})dt}-\frac{\delta}{2}$.

Fix such an $\varepsilon$. For every $u\in \mathcal{E}^+$. If $x>x_0(t)$, $U^{-}(t, x)=\psi^{\varepsilon}(t, x)>0$ and $\p_{x}U^-(t, x)>0$ for any $t\in\R$. By Lemma \ref{V-estimat} and \eqref{ignit-TV}, we have
\begin{align*}
&\mathcal{A}_u(\psi^{\varepsilon}(\cdot, \cdot))(t, x)\cr
&=\p_{xx}\psi^{\varepsilon}+(c-\chi\p_{x}\Psi(t, x;u))\p_{x}\psi^\varepsilon+\psi^{\varepsilon}(r(t,x)-\chi\lambda\Psi(t, x;u)-(b-\chi\mu)\psi^{\varepsilon})-\p_{t}\psi^{\varepsilon}\cr
&\ge \p_{xx}\psi^{\varepsilon}-c_{\varepsilon}\p_{x}\psi^{\varepsilon}+(c-\chi\p_{x}\Psi(t, x;u)+c_{\varepsilon})\p_{x}\psi^\varepsilon\cr
&\,\,\,\,\,\,+\psi^{\varepsilon}(r(t,\infty)-\varepsilon-\chi\lambda\Psi(t, x;u)-(b-\chi\mu)\psi^{\varepsilon})-\p_{t}\psi^{\varepsilon}\cr
&\geq \p_{xx}\psi^{\varepsilon}-c_{\varepsilon}\p_{x}\psi^{\varepsilon}+(c-\frac{\chi\mu \bar r }{2\sqrt{\lambda}(b-\chi\mu)}+2\sqrt{\frac{1}{\omega}\int_{0}^{\omega}(r(t, \infty)-\frac{\chi\mu\bar r}{b-\chi\mu})dt}-\frac{\delta}{2})\p_{x}\psi^\varepsilon\cr
&\,\,\,\,\,\,+\psi^{\varepsilon}(r(t, \infty)-\varepsilon-\frac{\chi\mu\bar r}{b-\chi\mu}-(b-\chi\mu)\psi^{\varepsilon})-\p_{t}\psi^{\varepsilon}\cr
&=(c-\frac{\chi\mu \bar r }{2\sqrt{\lambda}(b-\chi\mu)}+2\sqrt{\frac{1}{\omega}\int_{0}^{\omega}(r(t, \infty)-\frac{\chi\mu\bar r}{b-\chi\mu})dt}-\frac{\delta}{2})\p_{x}\psi^\varepsilon\cr
&> 0.
\end{align*}
If $x<x_0(t)$, $U^{-}(t, x)=0$. Then $\mathcal{A}_u(U^-)(t, x)=0$.

Since $x_1<x_0(t)$ for $t\in\R$, it is clear that $U^{-}(t, x)<U^{+}(x)$ for $t\in\R$, $x\in\R$. 
\end{proof}

%\textbf{(W.S. May consider
%$$
%f_{\epsilon,M}=u(r^*-\epsilon-M u^{\frac{1}{p}}-M u).
%$$
%Also consider
%\begin{equation*}
%\begin{cases}
%u_t=u_{xx}+\tilde c u_x+ f_{\varepsilon,M}(u)\cr
%u(t,0)=0
%\end{cases}
%\end{equation*}
%For $-2\sqrt{r^*-\epsilon}<\tilde c<0$, this equation has a stationary solution
%$u_{\epsilon,M}(x)$ which is monotone increasing. Choose $M\gg 1$. Then
%$$
%u_{\epsilon,M}(x)\ll 1.
%$$
%For given $c>-2\sqrt {r^*}$, choose $-2\sqrt{r^*-\epsilon}<\tilde c<\min\{c,0\}$ and $\epsilon\ll 1$, $M\gg 1$.
%Prove for any $u_0(x)$ with $u_{\epsilon,M}\le u_0(x)\le U^+(x)$,
%$$
%u_{\epsilon,M}(x)\le u(t,x;u_0)\le U^+(x)\quad \forall\, t\ge 0, \,\, x\in\R,
%$$
%where $u(t,x;u_0)$ is the solution of
%$$
%u_t=u_{xx}+(c-\chi v_x) u_x+ u\big( r(x)-\chi \lambda v-(b-\chi\mu)u\big),\quad x\in\R
%$$
%and
%$$
%v_{xx}-\lambda v+\mu u=0.
%$$
%Then use fixed point theorem to prove that there is $U^*(x)$ such that
%$$
%u_{\epsilon,M}\le U^*(x)\le U^+(x)
%$$
%and $U^*(x)$ is the forced wave solution.)
%}

 Let
$$
\mathcal{E}:=\{u\in \mathcal{E}^{+} : U^-(t, x)\leq u(t, x)\leq U^{+}(x), \,\ \forall\,\ t\in\R,\,\ x\in\R\}.
$$
For any $u\in \mathcal{E}$, $t_0\in\R$,
let $U(t,x;t_0, U^{+}, u)$ be the solution of the following parabolic equation
\begin{equation}\label{U-def}
\begin{cases}
\mathcal{A}_{u}(U)=0,\quad\ t>t_0,\,\,  x\in\R,\cr
U(t_0,x;t_0, U^{+}, u)=U^{+}(x).
\end{cases}
\end{equation}

\begin{lem}\label{limit-of-U-lm-1}
For any $u\in \mathcal{E}$, ${{U^*(t, x;u)}}:=\lim_{t_0\to-\infty}U(t,x;t_0, U^{+}, u)$ exists for any $t\in\R$, $x\in\R$ and satisfies the following parabolic equation
\begin{equation}\label{U-elliptic-eq}
\p_{t}U=\p_{xx}U+(c-\chi\p_{x}\Psi(\cdot,\cdot;u))\p_{x}U+(r(t, x)-\chi\lambda\Psi(\cdot,\cdot;u)-(b-\chi\mu)U)U, \,\ \forall\,\ (t, x)\in\R\times\R.
\end{equation}
Moreover, ${U^*(\cdot, \cdot;u)}\in\mathcal{E}$.
\end{lem}

\begin{proof}
First, thanks to  Lemma \ref{super-solu-lm},  it follows from the comparison principle that
$$
U(t,x;t_0, U^{+}, u)\leq U^+(x),\quad \forall \,\, x\in\R,\ t\geq t_0, \ u\in\mathcal{E}.
$$
Hence, for any $t_2<t_1$,
$$
U(t_1,x;t_2, U^{+}, u)\leq U^+(x)=U(t_1,x;t_1, U^{+}, u),\quad \forall \,\, x\in\R, \ u\in\mathcal{E}.
$$
By the comparison principle again, we then have
$$
U(t,x;t_2, U^{+}, u)\leq U(t,x;t_1, U^{+}, u),\quad \forall \,\, x\in\R,\,\ t\geq t_1, \ u\in\mathcal{E}.
$$
Thus, for every $u\in\mathcal{E}$, $U(t,x;t_0, U^{+}, u)$ is non-decreasing in $t_0\in\R$. Then for each $u\in\mathcal{E}$, there is $U^*(t, x;u)$ such that
\begin{equation}\label{U^*-main-def}
U^*(t, x;u):=\lim_{t_0\to-\infty}U(t,x;t_0, U^{+}, u),\quad \forall\ x\in\R,\,\ t\in\R.
\end{equation}
is well defined.
Moreover, by a priori  estimates for parabolic equations, it is not difficult to see that $U^*(t,x;u)$ satisfies \eqref{U-elliptic-eq}. In view of  the time periodicity of $r(t,x)$, we obtain
\begin{align*}
U^*(t, x;u)&=\lim_{t_0\to-\infty}U(t,x;t_0, U^{+}, u)\cr
&=\lim_{n\to\infty}U(t,x;-(n+1)\omega, U^{+}, u)\cr
%&=\lim_{n\to\infty}U(t+n\omega,x;0, U^{+}, u)\cr
%&=\lim_{n\to\infty}U(t+(n+1)\omega,x;0, U^{+}, u)\cr
&=\lim_{n\to\infty}U(t+\omega,x;-n\omega, U^{+}, u)\cr
&=U^*(t+\omega, x;u).
\end{align*}

Next, it follows from Lemma \ref{subsolution} and the comparison principle that
\begin{equation}\label{lower-bound of U}
U^{-}(t, x)\leq U(t,x;t_0, U^{+}, u),\quad \quad \forall \,\,x\in\R,\ t\ge t_{0}, \ u\in\mathcal{E}.
\end{equation}
Hence,
\begin{equation}
U^{-}(t, x)\leq { U^*(t, x;u)}\leq U^{+}(x),\quad \forall\ x\in\R,\ t\in\R,\ u\in\mathcal{E}.
\end{equation}
 Therefore, $U^*(\cdot, \cdot;u)\in\mathcal{E}$.
\end{proof}

\begin{lem}
\label{u-star-star-lm1}
For any $u\in\mathcal{E}$, suppose that $U_{*}(t,x;u)$ is also a solution of \eqref{U-elliptic-eq} in $\mathcal{E}$.
Then
\begin{equation*}
\label{new-add-eq1}
\lim_{x\to\infty} \frac{U_{*}(t,x;u)}{U^*(t,x;u)}=1
\end{equation*}
uniformly in $t\in\R$.
\end{lem}

\begin{proof}
First of all,  Lemma \ref{V-estimat} implies that  for any $u\in\mathcal{E}$, 
$$\sup_{t\in\R, x\in\R}|\Psi(t,x;u)|<\infty\quad {\rm and} \quad \sup_{t\in\R, x\in\R}|\frac{\p}{\p x}\Psi(x;u)|<\infty.$$
Since  $\Psi_{xx}(t,x;u)=\lambda \Psi(t,x;u)-\mu u(t,x)$, it follows that  for any $u\in\mathcal{E}$, 
$$\sup_{t\in\R, x\in\R}|\Psi_{xx}(t,x;u)|<\infty.$$ 
This implies that, without loss of generality, there are sequences $\{x_n\}_{n=1}^{\infty}\subset\R$, $\{t_n\}_{n=1}^{\infty}\subset\R$ such that $\lim_{n\to\infty} \Psi(t+t_{n}, x+x_{n};u)$ and $\lim_{n\to\infty}\Psi_x(t+t_{n}, x+x_{n};u)$  exist locally uniformly on $\R\times\R$.

Since  $\frac{U_{*}(t,x;u)}{U^*(t,x;u)}\le 1$ for all $t\in\R$, $x\in\R$,  it suffices to prove that
$$\liminf_{x\to \infty}\sup_{t\in\R}\frac{U_{*}(t,x;u)}{U^*(t,x;u)}\ge 1.$$ 
Assume, by contraction,  that
$$\liminf_{x\to \infty}\sup_{t\in\R}\frac{U_{*}(t,x;u)}{U^*(t,x;u)}<1.
$$
Then there are $0<\delta<1$ and  $x_n\to\infty$, $t_{n}\in\R$ such that
$$
\frac{U_{*}(t_{n}, x_n;u)}{U^*(t_{n}, x_n;u)}\le 1-\delta \quad \forall\, \, n=1,2,\cdots.
$$
Let
$$
U_{n,*}(t, x;u)=U_{*}(t+t_{n}, x+x_n;u),\,\ U_{n}^*(t, x;u)=U^*(t+t_{n}, x+x_n;u),
$$ 
and $ \Psi_n(t, x;u)=\Psi(t+t_{n}, x+x_n;u)$.
Without loss of generality, we may assume that $$\lim_{n\to\infty}(t_n-[\frac{t_n}{\omega}]\omega)=t^*$$
 and there are $\underline{U}_*(t, x;u)$, $\overline{U}^*(t, x;u)$ and $\Psi^*(t,x;u)$ such that
$$
\lim_{n\to \infty} U_{n,*}(t, x;u)=\underline{U}_*(t,x;u),\quad \lim_{n\to\infty}U_{n}^*(t,x;u)=\overline{U}^*(t,x;u),
$$
and $\lim_{n\to\infty}\Psi_n(t,x;u)=\Psi^*(t,x;u)$
locally uniformly on $\R\times\R$.
This implies that both $\underline{U}_*(t,x;u)$ and $\overline{U}^*(t,x;u)$ are solutions of the following equation
\begin{equation}\label{U-star-elliptic-eq}
U_{t}=U_{xx} +(c -\chi\Psi^*_x(t,x;u))U_x+(r(t+t^*, \infty)-\chi\lambda\Psi^*(t,x;u)-(b-\chi\mu)U)U.
\end{equation}

We now claim that $\underline{U}_*(t,x;u)\equiv\overline{U}^*(t,x;u)$. Indeed, since
$$
0<\inf_{(t, x)\in\R\times\R}\underline{U}_*(t,x;u)\le \inf_{(t, x)\in\R\times\R} \overline{U}^*(t,x;u)
$$
and
$$ \sup_{(t, x)\in\R\times\R}\underline{U}_*(t,x;u)\le \sup_{(t, x)\in\R\times\R}\overline{U}^*(t,x;u)<\infty,
$$
it follows  that the set 
 $$
\left \{\gamma\geq1 : \frac{1}{\gamma} \overline{U}^*(t, x;u)\le \underline{U}_*(t,x;u)\le \gamma \overline{U}^*(t,x;u),\quad \forall\,\,(t, x)\in\R\times\R\right\}
 $$
 is not empty.  Thus,  we can define
 $$
 \rho(\underline{U}_*,\overline{U}^*):=\inf\left\{\ln \gamma : \frac{1}{\gamma} \overline{U}^*(t,x;u)\le \underline{U}_*(t,x;u)\le \gamma \overline{U}^*(t,x;u), \quad \forall\,\,(t, x)\in\R\times\R \right\}.
$$
%Moreover, there is $\gamma\geq 1$ such that
%$$
%\rho(\underline{U}_*,\overline{U}^*)=\ln \gamma\,\ {\rm and}\,\, \frac{1}{\gamma} \overline{U}^*(x;u)\le \underline{U}_*(x;u)\le \gamma %\overline{U}^*(x;u)\quad \forall\,\, x\in\R.
%$$
Note that $\rho(\underline{U}_*,\overline{U}^*)$ is the so called {\it part metric} between $\underline{U}_*$ and $\overline{U}^*$.
Assume that $\rho(\underline{U}_*,\overline{U}^*)>0$. 
Then, there is $\gamma>1$ such that
$$
\rho(\underline{U}_*,\overline{U}^*)=\ln \gamma\,\ {\rm and}\,\,\, \frac{1}{\gamma} \overline{U}^*(t,x;u)\le \underline{U}_*(t,x;u)\le \gamma \overline{U}^*(t,x;u)\quad \forall\,\, x\in\R,\, t\in\R.
$$
Note that
$$
(\gamma \overline{U}^*)_{t}>(\gamma \overline{U}^*)_{xx}+(c -\chi\Psi^*_x(t,x;u))(\gamma \overline{U}^*)_{x}+(r(t+t^*, \infty)-\chi\lambda\Psi^*(t,x;u)-(b-\chi\mu)(\gamma \overline{U}^*))(\gamma \overline{U}^*),
$$
$$
(\gamma \underline{U}_*)_{t}>(\gamma \underline{U}_*)_{xx}+(c -\chi\Psi^*_x(t,x;u))(\gamma \underline{U}_*)_{x}+( r(t+t^*, \infty)-\chi\lambda\Psi^*(t,x;u)-(b-\chi\mu)(\gamma \underline{U}_*))(\gamma \underline{U}_*).
$$
By the comparison principle, we have
\begin{equation}\label{underU-overU-compare}
\begin{cases}
\underline{U}_*(t+t_{0},x;u)\leq U(t+t_0,x;t_0,\gamma \overline{U}^*,u)<\gamma \overline{U}^*(t+t_{0},x;u)\quad \forall\,x\in\R,\, t>0,t_0\in\R\cr
\overline{U}^*(t+t_{0}, x;u)\leq U(t+t_{0},x;t_0,\gamma \underline{U}_*,u)<\gamma \underline{U}_*(t+t_{0},x;u)\quad \forall\,x\in\R,\, t>0,t_0\in\R
\end{cases}
\end{equation}
where $U(t+t_0,x;t_0,\gamma \overline{U}^*,u)$ is the solution of
\begin{equation}\label{New-U-def}
\begin{cases}
U_t=U_{xx} +(c -\chi\Psi^*_x(t+t_0,x;u))U_x\cr
\quad \quad +( r(t+t_0+t^*,\infty)-\chi\lambda\Psi^*(t+t_0, x;u)-(b-\chi\mu)U)U,\quad t>0,\,\,  x\in\R,\cr
U(t_0,x;t_0,\gamma \overline{U}^*,u)=\gamma \overline{U}^*(t_0,x;u).
\end{cases}
\end{equation}
Let $l=\min\{\inf_{(t,x)\in\R\times\R}\overline{U}^*(t,x;u), \inf_{(t,x)\in\R\times\R}\underline{U}_*(t,x;u)\}=\inf_{(t,x)\in\R\times\R}\underline{U}_*(t,x;u)>0.$ It then follows  that
\begin{align}\label{overline{U}^*}
(\gamma \overline{U}^*)_{t}=&(\gamma \overline{U}^*)_{xx}+(c -\chi\Psi^*_x(t,x;u))(\gamma \overline{U}^*)_{x}+({ r(t+t^*,\infty)}-\chi\lambda\Psi^*(t,x;u)-(b-\chi\mu)\overline{U}^*)(\gamma \overline{U}^*)\cr
=&(\gamma \overline{U}^*)_{xx}+(c -\chi\Psi^*_x(t,x;u))(\gamma \overline{U}^*)_{x}+({ r(t+t^*,\infty)}-\chi\lambda\Psi^*(t,x;u)-(b-\chi\mu)\gamma\overline{U}^*)(\gamma \overline{U}^*)\cr
&+(b-\chi\mu)(\gamma-1)\overline{U}^*(\gamma \overline{U}^*)\cr
\geq & (\gamma \overline{U}^*)_{xx}+(c -\chi\Psi^*_x(t,x;u))(\gamma \overline{U}^*)_{x}+({ r(t+t^*,\infty)}-\chi\lambda\Psi^*(t,x;u)-(b-\chi\mu)\gamma\overline{U}^*)(\gamma \overline{U}^*)\cr
&+(b-\chi\mu)(\gamma-1)l(\gamma \overline{U}^*).
\end{align}
Similarly, we have
\begin{align}
(\gamma \underline{U}_*)_{t}\geq &(\gamma \underline{U}_*)_{xx}+(c -\chi\Psi^*_x(t,x;u))(\gamma \underline{U}_*)_{x}\nonumber\\
&+( r(t+t^*,\infty)-\chi\lambda\Psi^*(t,x;u)-(b-\chi\mu)\gamma\underline{U}_*)(\gamma \underline{U}_*)
+(b-\chi\mu)(\gamma-1)l(\gamma \underline{U}_*).
\end{align}
For any $\varepsilon>0$, set $w^*(t,x)=e^{\varepsilon t}U(t+t_0,x;t_0,\gamma \overline{U}^*,u)$. It follows from \eqref{New-U-def} that
\begin{align*}
w^*_{t}=&\varepsilon w^*+e^{\varepsilon t}U_{t}(t+t_0,x;t_0,\gamma \overline{U}^*,u)\cr
=&\varepsilon w^*+w^*_{xx}+(c -\chi\Psi^*_x(t+t_0, x;u))w^*_{x}\cr
&+(r(t+t_0+t^*,\infty)-\chi\lambda\Psi^*(t+t_0, x;u)-(b-\chi\mu)w^*)w^*\cr
&+(b-\chi\mu)(e^{\varepsilon t}-1)U(t+t_0,x;t_0,\gamma \overline{U}^*,u)w^*\cr
\leq &\varepsilon w^*+w^*_{xx}+(c -\chi\Psi^*_x(t+t_0, x;u))w^*_{x}\cr
&+( r(t+t_0+t^*,\infty)-\chi\lambda\Psi^*(t+t_{0},x;u)-(b-\chi\mu)w^*)w^*\cr
&+(b-\chi\mu)(e^{\varepsilon t}-1)\gamma \overline{U}^*(t+t_0,x;u)w^*\cr
\leq &w^*_{xx}+(c -\chi\Psi^*_x(t+t_0, x;u))w^*_{x}\cr
&+( r(t+t_0+t^*,\infty)-\chi\lambda\Psi^*(t+t_{0},x;u)-(b-\chi\mu)w^*)w^*\cr
&+(\varepsilon+(b-\chi\mu)(e^{\varepsilon t}-1)\gamma L)w^*,
\end{align*}
where $L=\max\{\sup_{(t,x)\in\R\times\R}\overline{U}^*(t,x;u),\,  \sup_{(t,x)\in\R\times\R}\underline{U}_*(t,x;u)\}=\sup_{(t,x)\in\R\times\R}\overline{U}^*(t,x;u)<\infty$.

Choose $0<\varepsilon\ll 1$ such that
$$
\varepsilon+(b-\chi\mu)(e^{\varepsilon t}-1)\gamma L<(b-\chi\mu)(\gamma-1)l \quad \forall\,\ 0\leq t\leq 1.
$$
Thus, for $x\in\R$, $t_0\in\R$ and $0\leq t\leq 1$, we have
\begin{align}\label{w^*-eq}
w^*_{t}&\leq w^*_{xx}+(c -\chi\Psi^*_x(t+t_{0},x;u))w^*_{x}\cr
&\,\,\,\,\,\, +(r(t+t_0+t^*,\infty)-\chi\lambda\Psi^*(t+t_{0}, x;u)-(b-\chi\mu)w^*)w^*\cr
&\,\,\,\,\,\,+(b-\chi\mu)(\gamma-1)lw^*
\end{align}
By the comparison principle, \eqref{overline{U}^*} and \eqref{w^*-eq},
it follows  that
$$
w^*(t,x)=e^{\varepsilon t}U(t+t_0,x;t_0,\gamma \overline{U}^*,u)\leq \gamma \overline{U}^*(t+t_0, x;u) \quad \forall\,\ x\in\R, \,\, 0\leq t\leq 1,\,\, t_0\in\R
$$
for $0<\varepsilon\ll 1$. Hence,  there is $0<\varepsilon_0\ll 1$ such that
$$
U(1+t_0,x;t_0,\gamma \overline{U}^*,u)\leq e^{-\varepsilon_0}\gamma \overline{U}^*(1+t_0, x;u)\quad \forall\,\ x\in\R,\,\ t_0\in\R.
$$
Combining this with \eqref{underU-overU-compare}, we get
\begin{equation*}\label{underline{U}_*}
\underline{U}_*(1+t_0, x;u)\leq e^{-\varepsilon_0}\gamma \overline{U}^*(1+t_0,x;u)\quad \forall\,\ x\in\R, \,\ t_0\in\R,
\end{equation*}
which implies that
\begin{equation}\label{underline{U}_*}
\underline{U}_*(t, x;u)\leq e^{-\varepsilon_0}\gamma \overline{U}^*(t,x;u)\quad \forall\,\ x\in\R, \,\ t\in\R,
\end{equation}
Set $w_*(t,x)=e^{\varepsilon t}U(t+t_{0},x;t_0,\gamma \underline{U}_*,u)$. 
By the similar procedure, we can show that
\begin{equation}\label{overline{U}_*}
\overline{U}^*(t,x;u)\leq e^{-\varepsilon_0}\gamma \underline{U}_*(t,x;u)\quad \forall\,\ x\in\R,\,\ t\in\R.
\end{equation}
In view of  \eqref{underline{U}_*} and \eqref{overline{U}_*}, it follows that
$$
\frac{1}{e^{-\varepsilon_0}\gamma}\overline{U}^*(t,x;u)\leq \underline{U}_*(t,x;u)\leq e^{-\varepsilon_0}\gamma \overline{U}^*(t,x;u)\quad \forall\,\ x\in\R,\,\ t\in\R.
$$
By the definition of  $\rho(\underline{U}_*,\overline{U}^*)$, we obtain 
$\gamma\leq e^{-\varepsilon_0}\gamma$, 
which is a contradiction. It follows that $\rho(\underline{U}_*,\overline{U}^*)=0$, and hence, $\underline{U}_*(t,x;u)=\overline{U}^*(t,x;u)$ for all $x\in\R$, $t\in\R$.
By the assumption, we then have 
$$
\underline{U}_*(0,0;u)\not=\overline{U}^*(0,0;u),
$$
a contradiction.
\end{proof}

\begin{lem}
\label{u-star-star-lm2}
For any $u\in\mathcal{E}$, $U^*(t,x;u)$ is the unique positive solution of \eqref{U-elliptic-eq} in $\mathcal{E}$.
\end{lem}

\begin{proof}
Suppose that $U_{*}(t,x;u)$ is any positive solution of \eqref{U-elliptic-eq} in $\mathcal{E}$. It suffices to prove that
$U_{*}(t,x;u)\equiv U^*(t,x;u)$.

For any $\epsilon>0$, let
$$
K_\epsilon=\{k\ge 1 : kU_{*}(t,x;u)\ge U^*(t,x;u)-\epsilon\quad \forall \,\, x\in\R, \,\ t\in\R\}.
$$
By Lemma \ref{u-star-star-lm1} and the fact $\lim_{x\to -\infty}\sup_{t\in\R}U_{*}(t,x;u)=\lim_{x\to -\infty}\sup_{t\in\R}U^*(t,x;u)=0$, it follows that $K_\epsilon\not =\emptyset$.
Let
$$
k_\epsilon=\inf K_\epsilon.
$$
Then $k_\epsilon\ge 1$ and
\begin{equation}
\label{new-add-eq0}
k_\epsilon U_{*}(t,x;u)\ge U^*(t,x;u)-\epsilon\quad \forall\,\, x\in\R, \,\ t\in\R.
\end{equation}
For any $0<\epsilon_1<\epsilon_2$, since
$$
k_{\epsilon_1} U_{*}(t,x;u)\ge U^*(t,x;u)-\epsilon_1>U^*(t,x;u)-\epsilon_2\quad \forall\,\, x\in\R, \,\ t\in\R,
$$
it then follows that $k_{\epsilon_1}\geq k_{\epsilon_2}$. Thus, $k_\epsilon$ is nonincreasing in $\epsilon>0$.
If $k_{\epsilon}=1$ for any $\epsilon>0$, clearly, we have that ${U^*(t,x;u)}\equiv {U_{*}(t,x;u)}.$

Assume that there exists $\epsilon_0>0$ such that $k_{\epsilon_0}>1$. Then
\begin{equation}
\label{new-add-eq2}
k_{\epsilon}\geq k_{\epsilon_0}>1 \quad \text{for any}\quad  0<\epsilon\leq \epsilon_0.
\end{equation}
For any given $0<\epsilon\leq \epsilon_0$, since $k_{\epsilon}>1$, there exists $\delta>0$, such that $k_{\epsilon}-\delta>1$.
In view of Lemma \ref{u-star-star-lm1}, it follows that for such given $\epsilon>0$,
$$
\frac{U_{*}(t,x;u)}{U^*(t,x;u)}\geq 1-\frac{\epsilon}{U^*(t,x;u)}, \quad x\gg 1,\,\ t\in\R.
$$
Hence,
\begin{equation}\label{x>>1}
(k_{\epsilon}-\delta)U_{*}(t,x;u)\geq U^*(t,x;u)-\epsilon, \quad x\gg 1,\,\ t\in\R.
\end{equation}
Since
$\lim_{x\to-\infty}\frac{U^*(t,x;u)-\epsilon}{U_{*}(t,x;u)}=-\infty$ uniformly in $t\in\R$,
we have
\begin{equation}\label{x<<-1}
(k_{\epsilon}-\delta)U_{*}(t,x;u)\geq U^*(t,x;u)-\epsilon, \quad x\ll -1, \,\ t\in\R.
\end{equation}
It then follows from \eqref{x>>1}, \eqref{x<<-1} and the definition of $k_{\epsilon}$ that
there is $(t_\epsilon, x_\epsilon)\in\R\times\R$ such that
\begin{equation}
\label{new-add-eq3}
k_\epsilon U_{*}(t_\epsilon, x_\epsilon;u)=U^*(t_\epsilon, x_\epsilon;u)-\epsilon.
\end{equation}
In particular, the function $t\to k_\epsilon U_{*}(t, x_\epsilon;u)-U^*(t, x_\epsilon;u)+\epsilon$ attains
its minimum value at $t_\epsilon$, and the function $x\to k_\epsilon U_{*}(t_\epsilon, x;u)-U^*(t_\epsilon, x;u)+\epsilon$ attains its minimum value at $x_\epsilon$. Since $U_{*}$ and $U^{*}$ are periodic in $t$, without loss of generality, we may assume that $t_{\epsilon}\in [0,\omega]$. 
 By Lemma \ref{u-star-star-lm1}, $x_\epsilon$ is bounded above.

We claim that $x_\epsilon$ is bounded from below.  In fact, at $(t_\epsilon, x_\epsilon)$, we have
$$
\p_{xx}(k_\epsilon U_{*}(t_\epsilon, x_\epsilon;u)-U^*(t_\epsilon, x_\epsilon;u))\ge 0,\quad \p_x (k_\epsilon U_{*}(t_\epsilon, x_\epsilon;u)-U^*(t_\epsilon, x_\epsilon;u))=0,
$$
and
$$
\p_t (k_\epsilon U_{*}(t_\epsilon, x_\epsilon;u)-U^*(t_\epsilon, x_\epsilon;u))=0.
$$
and hence,
\begin{align*}
0&\ge  k_\epsilon U_{*}(t_\epsilon, x_\epsilon; u)(r(t_\epsilon, x_\epsilon )-\chi\lambda\Psi(t_\epsilon, x_\epsilon;u)-(b-\chi\mu)U_{*}(t_\epsilon, x_\epsilon;u))\\
&\,\ -U^*(t_\epsilon, x_\epsilon;u)(r(t_\epsilon, x_\epsilon )-\chi\lambda\Psi(t_\epsilon, x_\epsilon;u)-(b-\chi\mu)U^*(t_\epsilon, x_\epsilon; u))\\
&\ge k_\epsilon U_{*}(t_\epsilon, x_\epsilon; u)(r(t_\epsilon, x_\epsilon )-\chi\lambda\Psi(t_\epsilon, x_\epsilon;u)-(b-\chi\mu)U^*(t_\epsilon, x_\epsilon; u))\\
&\,\ -U^*(t_\epsilon, x_\epsilon;u)(r(t_\epsilon, x_\epsilon )-\chi\lambda\Psi(t_\epsilon, x_\epsilon;u)-(b-\chi\mu)U^*(t_\epsilon, x_\epsilon; u))\\
&=-\epsilon (r(t_\epsilon, x_\epsilon )-\chi\lambda\Psi(t_\epsilon, x_\epsilon;u)-(b-\chi\mu)U^*(t_\epsilon, x_\epsilon; u)).
\end{align*}
%\begin{align*}
%0&\ge  k_\epsilon U_{1*}(x_\epsilon;u)(r(x_\epsilon )-\chi\lambda\Psi(x_\epsilon;u)-(b-\chi\mu)U_{1*}(x_\epsilon;u))\\
%&-U_1^*(x_\epsilon;u)(r(x_\epsilon )-\chi\lambda\Psi(x_\epsilon;u)-(b-\chi\mu)U_1^*(x_\epsilon;u))\\
%&\ge k_\epsilon U_{1*}(x_\epsilon;u)(r(x_\epsilon )-\chi\lambda\Psi(x_\epsilon;u)-(b-\chi\mu)U_1^*(x_\epsilon;u))\\
%&-U_1^*(x_\epsilon;u)(r(x_\epsilon )-\chi\lambda\Psi(x_\epsilon;u)-(b-\chi\mu)U_1^*(x_\epsilon;u))\\
%&=-\epsilon (r(x_\epsilon )-\chi\lambda\Psi(x_\epsilon;u)-(b-\chi\mu)U_1^*(x_\epsilon;u)).
%\end{align*}
This implies that $r(t_\epsilon, x_\epsilon)\ge 0$ and hence $x_\epsilon$ is bounded from below.

Therefore,  $x_{\epsilon}$ is bounded both from below and above.
By \eqref{new-add-eq3},
$$
k_\epsilon=\frac{U^*(t_\epsilon, x_\epsilon;u)-\epsilon}{U_{*}(t_\epsilon, x_\epsilon;u)}.
$$
Hence $k_\epsilon$ is bounded,  and there is $\epsilon_n\to 0$ such that
$x_{\epsilon_n}\to x^*$, $t_{\epsilon_n}\to t^*$ and $k_{\epsilon_n}\to k^*(\ge k_{\epsilon_0}>1)$ as $n\to\infty$.
This, together with \eqref{new-add-eq3}, implies that
$k^* U_{*}(t^*, x^*;u)=U^*(t^*, x^*;u).$
By \eqref{new-add-eq0},
$$
k^* U_{*}(t,x;u)\ge U^*(t,x;u)\quad \forall\,\, t\in\R,\,\  x\in \R.
$$
Since $k^*>1$, by Lemma \ref{u-star-star-lm1},  $k^* U_{*}(0,x;u)\not \equiv U^*(0,x;u)$.
Then the comparison principle implies that
$$
U^*(t,x;u)<k^*U_{*}(t,x;u) \quad \forall\,\ t>0,\,\ x\in\R,
$$
By the periodity of $U_{*}$ and $U^{*}$ in $t$, we have 
$$
U^*(t,x;u)<k^*U_{*}(t,x;u) \quad \forall\,\ t\in\R,\,\ x\in\R,
$$
which is a contraction. Therefore, $k_\epsilon= 1$ for $0<\epsilon\ll 1$
and then
$
U_{*}(t,x;u)\ge U^*(t,x;u)
$
for all $(t,x)\in\R\times\R$. By interchanging the role of $U_{*}$ and $U^{*}$, we also have $U_{*}(t,x;u)\le U^*(t,x;u)$ for all $(t,x)\in\R\times\R$. Thus, we obtain $U_{*}(t,x;u)\equiv U^*(t,x;u)$.
\end{proof}

We now prove Theorem \ref{forced-wave}.

\begin{proof}[Proof of Theorem \ref{forced-wave}]
First of all, let $C_{\rm unif}^{b}(\R\times\R)$ be endowed with the compact-open topology, that is, a sequence of elements of $C_{\rm unif}^{b}(\R\times\R)$ converges if and only if it converges uniformly on every compact subsets of $\R\times\R$. It is clear that $\mathcal{E}$ is a closed, bounded and convex subset of $C_{\rm unif}^{b}(\R\times\R)$ in the compact-open topology. Consider the mapping $U^*: \mathcal{E}\ni u \mapsto U^*(\cdot, \cdot;u)\in\mathcal{E}$ as defined by \eqref{U^*-main-def}.
Next, we prove that the mapping $U^*$ has a fixed point. We divide the proof into two steps.

\noindent{\bf Step 1.} We prove that the mapping $U^*: \mathcal{E}\ni u \mapsto U^*(\cdot, \cdot;u)\in\mathcal{E}$ is compact.

Indeed, first of all, let $\{u_{n}\}$ be a sequence of elements in $\mathcal{E}$. Since $U^*(\cdot, \cdot;u_{n})\in\mathcal{E}$ for every $n\geq 1$, then 
\begin{equation}\label{unif-bdd-U-seq}
\sup_{t\in\R, x\in\R, n\geq 1}|U^*(t, x;u_{n})|<\frac{\bar r}{b-\chi\mu},
\end{equation}
which implies that $U^*(\cdot, \cdot;u_{n})$ is uniformly bounded in compact-open topology.

Next, let $\{T(t)\}_{t\ge 0}$ be the $C_0-$semigroup generated by $\Delta-I$ on $C^b_{\rm unif}(\R^N)$.  Since $\partial_{xx}\Psi(t, x;u)=\lambda\Psi(t, x;u)-\mu u(t,x)$, the variation of constant formula yields that for every $n\geq 1$,
\begin{equation}
\begin{split}
&U(t, \cdot;t_0,U^+,u_{n})\\
%=&T(t-t_0)U^++ \int_{t_0}^tT(t-s)(c-\chi\p_x \Psi(s,\cdot;u_{n}))\p_xU(s,\cdot;t_0,U^+,u_{n})ds \\
%&+ \int_{t_0}^tT(t-s)(r(s,x)+1-\lambda\chi \Psi(s, \cdot;u_{n})-(b-\chi\mu)U(s,\cdot;t_0,U^+,u_{n}))U(s,\cdot;t_0,U^+,u_{n})ds\\
=& \underbrace{T(t-t_0)U^+}_{I^1_{t_0, n}(t, \cdot)} + \underbrace{\int_{t_0}^tT(t-s)\partial_x((c-\chi\p_x \Psi(s,\cdot,u_{n}))U(s,\cdot;t_0,U^+,u_{n}))ds}_{I^2_{t_0, n}(t, \cdot)} \\
&+ \underbrace{\int_{t_0}^tT(t-s)(r(s, x)+1-\chi\mu u_{n}(s,\cdot))U(s,\cdot;t_0,U^+,u_{n})ds}_{I^3_{t_0, n}(t, \cdot)}\\
&-\underbrace{\int_{t_0}^tT(t-s)(b-\chi\mu)U^2(s,\cdot;t_0,U^+,u_{n})ds}_{I^4_{t_0,n}(t, \cdot)}.
\end{split}
\end{equation}
We then choose $0<\beta<\frac{1}{2}$. Let $X^{\beta}$ be the fractional power space associated with $\Delta-I$ on $C^{b}_{\rm unif}(\R^N)$.
We  claim that there is a constant $\tilde{C}_{\beta}$ independent of $t$, $n$, and $t_0$ such that
\begin{equation}\label{b-eq1}
\|U(t, \cdot;t_0, U^+,u_{n})\|_{X^{\beta}}\leq \tilde{C}_{\beta}\Big(\frac{e^{-(t-t_0)}}{(t-t_0)^{\beta}}+1\Big),\quad  \forall\,\ t>t_0.
\end{equation}
In fact,
by \cite[Inequalities (2.10), (2.12) and Lemma 3.2]{SaSh1} and Lemma \ref{V-estimat}, there is a constant $C_{\beta}>0$ such that
\begin{equation}\label{1-eq}
\|I^1_{t_0,n}(t,\cdot)\|_{X^\beta}\leq C_{\beta}(t-t_0)^{-\beta}e^{-(t-t_0)}\|U^+\|_{\infty}=\frac{C_{\beta}e^{-(t-t_0)}\bar r}{(t-t_0)^{\beta}(b-\chi\mu)},
\end{equation}
\begin{align}\label{2-eq}
\|I_{t_0,n}^2(t,\cdot)\|_{X^\beta}&\leq  C_{\beta}\int_{t_0}^{t}(\frac{t-s}{2})^{-(\frac{1}{2}+\beta)}e^{-(t-s)}(|c|
+\chi\|\p_{x}\Psi(s,\cdot;u_{n})\|_{\infty})\|U(s,\cdot;t_0,U^+,u_{n})\|_{\infty}ds\nonumber\\
&\leq C_{\beta}\int_{t_0}^{t}(\frac{t-s}{2})^{-(\frac{1}{2}+\beta)}e^{-(t-s)}(|c|+\frac{\chi\mu\bar r}{2\sqrt{\lambda}(b-\chi\mu)})\frac{\bar r}{b-\chi\mu}ds
\nonumber\\
&\leq  C_{\beta}\Big(|c|+\frac{\chi\mu\bar r}{2\sqrt{\lambda}(b-\chi\mu)}\Big)\frac{\bar r\Gamma(\frac{1}{2}-\beta)}{b-\chi\mu},
\end{align}
\begin{align}\label{3-eq}
\|I^3_{t_0,n}(t,\cdot)\|_{X^\beta}&\leq
 C_{\beta}\int_{t_0}^t(t-s)^{-\beta}e^{-(t-s)}\Big(\bar r+1+\chi\mu\|u_{n}(s,\cdot)\|_{\infty} \Big)\|U(s,\cdot;t_0,U^+,u_{n})\|_{\infty}ds\nonumber\\
&\leq  C_{\beta}\Big(\bar r+1+\frac{\chi\mu \bar r}{b-\chi\mu}\Big)\frac{\bar r\Gamma(1-\beta)}{b-\chi\mu},
\end{align}
and
\begin{align}\label{4-eq}
\|I^4_{t_0,n}(t,\cdot)\|_{X^{\beta}}&\leq C_{\beta}\int_{t_0}^t(t-s)^{-\beta}e^{-(t-s)}(b-\chi\mu)\|U(s,\cdot;t_0,U^+,u_{n})\|^{2}_{\infty}ds\cr
&\leq C_{\beta}\frac{\bar r^2\Gamma(1-\beta)}{b-\chi\mu}.
\end{align}
The claim then follows from \eqref{1-eq}--\eqref{4-eq}.

Now, we claim that there is a constant $\bar{C}_{\beta}>0$ independent of $t_0$, $n$, and $t$ such that
\begin{equation}\label{b-eq2}
\|U(t+h,\cdot,;t_0,U^+,u_{n})-U(t,\cdot,;t_0,U^+,u_{n})\|_{X^{\beta}}\leq \bar{C}_{\beta}(h^{\beta}+h^{1-\beta}+h^{\frac{1}{2}-\beta})\Big(\frac{e^{-(t-t_0)}}{(t-t_0)^{2\beta}}+1 \Big)
\end{equation}
for $h>0$ and $t>t_0$.
In fact, for $h>0$ and $t>t_0$, we have
\begin{equation}\label{1'-eq}
\|I_{t_0, n}^1(t+h,\cdot)-I_{t_0,n}^1(t, \cdot)\|_{X^\beta}=\|(T(h)-I)T(t-t_0)U^+\|_{X^{\beta}}\leq \frac{C_{\beta}\bar rh^{\beta}e^{-(t-t_0)}}{(b-\chi\mu)(t-t_0)^{2\beta}},
\end{equation}
\begin{align}\label{2'-eq}
&\|I_{t_0,n}^2(t+h, \cdot)-I_{t_0,n}^2(t, \cdot)\|_{X^\beta}\nonumber\\
&\le   C_{\beta}h^{\beta}\int_{t_0}^t\frac{e^{-(t-s)}}{(t-s)^{\frac{1}{2}+2\beta}}\Big(|c|
+\chi\|\p_{x}\Psi(s,\cdot;u_{n})\|_{\infty}\Big)\|U(s, \cdot;t_0,U^+,u_{n})\|_{\infty}ds\nonumber\\
& \,\, + C_{\beta}\int_{t}^{t+h}\frac{e^{-(t+h-s)}}{(t+h-s)^{\frac{1}{2}+\beta}}\Big(|c|
+\chi\|\p_{x}\Psi(s,\cdot;u_{n})\|_{\infty}\Big)\|U(s, \cdot;t_0,U^+,u_{n})\|_{\infty}ds\nonumber\\
&\leq  C_{\beta}\Big(h^{\beta}\Gamma(\frac{1}{2}-2\beta)+\frac{h^{\frac{1}{2}-\beta}}{\frac{1}{2}-\beta} \Big)\Big(|c|+\frac{\chi\mu\bar r}{2\sqrt{\lambda}(b-\chi\mu)}\Big)\frac{\bar r}{b-\chi\mu},
\end{align}
\begin{equation}\label{3'-eq}
\|I_{t_0,n}^3(t+h, \cdot)-I_{t_0,n}^3(t, \cdot)\|_{X^\beta}\leq  C_{\beta}\Big(h^{\beta}\Gamma(1-2\beta)+\frac{h^{1-\beta}}{1-\beta} \Big)\Big(\bar r+1+\frac{\chi\mu \bar r}{b-\chi\mu} \Big)\frac{\bar r}{b-\chi\mu},
\end{equation}
and
\begin{equation}\label{4'-eq}
\|I_{t_0,n}^4(\cdot,t+h)-I_{t_0,n}^4(\cdot,t)\|_{X^\beta}\leq   C_{\beta}\Big(h^{\beta}\Gamma(1-2\beta)+\frac{h^{1-\beta}}{1-\beta} \Big)\frac{\bar r^{2}}{b-\chi\mu}.
\end{equation}
Thus, \eqref{b-eq2} follows from \eqref{1'-eq}--\eqref{4'-eq}.

Letting $t_0\to -\infty$ in \eqref{b-eq2},
we obtain
\begin{equation}\label{b-eq2-1}
\|U^{*}(t+h,\cdot;u_{n})-U^{*}(t,\cdot;u_{n})\|_{X^{\beta}}\leq \bar{C}_{\beta}(h^{\beta}+h^{1-\beta}+h^{\frac{1}{2}-\beta}),  \ \forall\,\ h>0,\ t\in\R.
\end{equation}
Note that $X^{\beta}$ is continuously embedded in $C^{\nu}(\R)$ for $0<\nu<2\beta$ (\cite{Hen}). This together with \eqref{b-eq1} implies that there is $\hat C_\beta>0$ such that
\begin{equation}
\label{b-eq1-1}
|U^{*}(t, x+h;u_{n})-U^{*}(t,x;u_{n})|\le \hat C_\beta |h|^{\nu}\quad \forall \,\, x\in\R,\, t\in\R,\, h>0,\, n\ge 1.
\end{equation}

Finally, by  \eqref{unif-bdd-U-seq}, \eqref{b-eq2-1},   \eqref{b-eq1-1}, Arzela-Ascoli Theorem,  and \cite[Chapter 3 - Theorem 15]{Frie},  there is a function $U^{*}(t,x)\in C^{1,2}(\R\times\R)$ and a subsequence $\{u_{n_{k}}\}$ of $\{u_{n}\}$ such that $U^{*}(t,x;u_{n_{k}})$ converges locally uniformly to $U^{*}(t,x)$ in $C^{1,2}(\R\times\R)$.

\noindent{\bf Step 2.} We prove that the mapping $U^*: \mathcal{E}\ni u \mapsto U^*(\cdot, \cdot;u)\in\mathcal{E}$ is continuous.

Indeed, let $\{u_{n}\}_{n\ge 1}$ and $u$ be elements of $\mathcal{E}$ such that $u_{n}\to u$ as $n\to\infty$
uniformly on every compact subset of $\R\times\R$. Suppose by contradiction that $U^*(t, x; u_{n})$ does not converge to $U^*(t, x; u)$ as $n\to\infty$ on every compact subset of $\R\times\R$. Thus, there is $\delta>0$ and a subsequence $\{u_{n1}\}_{n\geq 1}$ of $\{u_{n}\}_{n\geq 1}$ such that
\begin{equation}\label{U-star-nc}
\sup_{(t, x)\in [-K, K]^{2}}|U^*(t, x; u_{n1})-U^*(t, x; u)|\geq \delta, \quad \forall\, n\geq 1,\,\  K>0.
\end{equation}

Since $\mathcal{E}$ is compact, there is a subsequence
$\{u_{n2}\}_{n\geq 1}$ of $\{u_{n1}\}_{n\geq 1}$ and a function $U^*(\cdot,\cdot)\in \mathcal{E}$ such that $U^*(t, x; u_{n_{2}})\to U^*(t,x)$ as $n\to\infty$ on every compact subset of $\R\times\R$.
Observe from Lemma \ref{new-lm1} that
\begin{equation}\label{V-}
\Psi(t, x;u_{n_{2}})=\frac{\mu}{2\sqrt{\lambda}}\int_{\R}e^{-\sqrt{\lambda}|z|}u_{n_{2}}(t, x-z)dz
\end{equation}
Hence,
\begin{align}\label{V-1}
|\Psi(t, x;u_{n_{2}})-\Psi(t, x;u)|&=|\frac{\mu}{2\sqrt{\lambda}}\int_{\R}e^{-\sqrt{\lambda}|z|}(u_{n_{2}}(t, x-z)-u(t,x-z))dz|\cr
&\leq \frac{\mu}{2\sqrt{\lambda}}\int_{-R}^{R} e^{-\sqrt{\lambda}|z|} |u_{n_{2}}(t, x-z)-u(t,x-z)|dz\cr
&\quad +\frac{\mu}{2\sqrt{\lambda}}\int_{\R \backslash[-R,R]} e^{-\sqrt{\lambda}|z|} |u_{n_{2}}(t, x-z)-u(t,x-z)|dz
\end{align}
On one hand, since $u_{n_{2}}\to u$ as $n\to\infty$ uniformly on every compact subset of $\R\times\R$, then
\begin{align}\label{V-2}
\lim_{n\to\infty}\sup_{(t,x)\in[-K,K]^2}\frac{\mu}{2\sqrt{\lambda}}\int_{-R}^{R} e^{-\sqrt{\lambda}|z|} |u_{n_{2}}(t, x-z)-u(t,x-z)|dz=0 \quad \forall\,\ K>0, \,\ R>0.
\end{align}
On the other hand, note that $u\leq \frac{\bar r}{b-\chi\mu}$ for any $u\in \mathcal{E}$,
\begin{align}\label{V-3}
\lim_{R\to\infty}\sup_{n\ge 1, t\in\R, x\in\R}\frac{\mu}{2\sqrt{\lambda}}\int_{\R \backslash[-R,R]} e^{-\sqrt{\lambda}|z|} |u_{n_{2}}(t, x-z)-u(t,x-z)|dz=0.
\end{align}
Therefore, it follows from \eqref{V-}--\eqref{V-3} that $\Psi(t, x;u_{n_{2}})\to\Psi(t, x;u)$ as $n\to\infty$ uniformly on every compact subset of $\R\times\R$.

By \eqref{space-derivative-of-psi-1} and  the similar arguments to the above, it follows that $\p_{x}\Psi(t, x;u_{n_{2}})\to\p_{x}\Psi(t, x;u)$ as $n\to\infty$ uniformly on every compact subset of $\R\times\R$. Therefore,  \cite[Chapter 3, Theorem 15]{Frie} implies that $U^{*}(t, x)\in C^{1,2}(\R\times\R)$ and satisfies
$$
\p_{t}U=\p_{xx}U+(c-\chi\p_{x}\Psi(\cdot,\cdot;u))\p_{x}U+(r(t, x)-\chi\lambda\Psi(\cdot,\cdot;u)-(b-\chi\mu)U)U, \,\ \forall\, (t, x)\in\R\times\R.
$$
Therefore, it follows from Lemma \ref{u-star-star-lm2}  that $U^{*}(t, x)=U^{*}(t, x; u)$ for every $(t,x)\in\R\times\R$. But, letting $n\to\infty$ in \eqref{U-star-nc}, we have 
$$
\sup_{(t, x)\in [-K, K]^{2}}|U^*(t, x)-U^*(t, x; u)|\geq \delta, \quad \forall\,\, K>0,
$$
a contradiction. This shows that the mapping $U^{*}$ is continuous. 

By the Schauder fixed point theorem, there is $u^*\in \mathcal{E}$ such that $U^{*}(t,x; u^*)=u^*(t,x)$ for any $t\in\R$, $x\in\R$. 
Let $U(t,x)=u^*(t,x)$ and $V(t,x)=\Psi(t,x; U)$. Then $$(u(t,x), v(t,x))=(U(t, x-ct), V(t, x-ct))$$ is a time $\omega-$periodic entire solution of  \eqref{Keller-Segel-eq0-1}. Moreover, since $\lim_{x\to-\infty}U^{-}(t,x)=\lim_{x\to-\infty}U^{+}(x)=0$ uniformly in $t\in\R$, it yields that 
\begin{equation}\label{U-limit}
\lim_{x\to-\infty}U(t,x)=0  \quad {\rm uniformly}\,\, {\rm in}\,\, t\in\R.
\end{equation}

In the following we show that
\begin{equation}\label{kk-0}
\lim_{x\to\infty}|U(t, x)-p(t)|=0\quad {\rm uniformly}\,\, {\rm in}\,\ t\in\R.
\end{equation}
Suppose on the contrary that this is false. Then, there is a constant $\delta>0$ and a sequence $\{x_n\}$,
$x_{n}\to\infty$, a sequence $\{t_{n}\}$, $t_{n}\in\R$, such that
\begin{equation}\label{kkk-1}
|U(t_{n}, x_{n})-p(t_{n})|\geq \delta, \quad \forall\,\, n\geq 1.
\end{equation}
%Let $[\frac{t_n}{\omega}]$ be the greatest integer less than or equal to $\frac{t_n}{\omega}$. Then $t_{n}-[\frac{t_n}{\omega}]\omega\in [0, \omega)$. 
Since both $U(t,x)$ and $p(t)$ are periodic in $t$, 
without loss of generality, we may assume that $t_{n}\in [0, \omega]$ and there is $t^*\in [0, \omega]$ such that $t_{n}\to t^*$ as $n\to\infty$.
% $\lim_{n\to\infty}t_{n}-[\frac{t_n}{\omega}]\omega=t^*$. 
Consider the sequence of functions
$$
u_{n}(t,x)=u(t+t_{n},x+x_n+ct_{n})\quad \text{and}\quad v_{n}(t,x)=v(t+t_{n},x+x_n+ct_{n}).
$$
By a priori estimate for parabolic equations, without loss of generality, we may assume that there is $(u^{**}(t,x),v^{**}(t,x))\in C^{1,2}(\R\times\R)$ such that $(u_{n},v_{n})(t,x)\to$  $(u^{**}(t,x), v^{**}(t,x))$ {locally uniformly in $C^{1,2}(\R\times\R)$} as $n\to\infty$. Furthermore, the function $(u^{**}(t,x),v^{**}(t,x))$ is an entire solution of the following equation
\begin{equation}\label{final-eq}
\begin{cases}
u_t=u_{xx}-\chi(uv_x)_x +u(r(t+t^{*}, \infty)-bu),\quad t\in\R, \,\ x\in\R\cr
0=v_{xx}- \lambda v+\mu u,\quad t\in\R, \,\ x\in\R.
\end{cases}
\end{equation}
Note that
$$
u^{**}(t,x)=\lim_{n\to\infty}u_{n}(t,x)\geq \lim_{n\to\infty}U^{-}(t+t_{n}, x+x_n-ct)=p(t+t^{*})>0,\quad \forall\ x\in\R, \ t\in\R.
$$
So $\inf_{(t,x)\in\R\times\R}u^{**}(t,x)>0$. By the comparison principle, it follows  that $\inf_{(t,x)\in\R\times\R}v^{**}(t,x)>0$. Hence, $(u^{**}, v^{**})$ is a strictly positive entire solution of \eqref{final-eq}. By the uniqueness of strictly positive entire solutions of \eqref{final-eq} (\cite[Theorem 1.5]{SaSh4}), we obtain $u^{**}(t, x)=p(t+t^*)$ for any $t\in\R$, $x\in\R$. In particular, $p(t^*)=u^{**}(0,0)$. But \eqref{kkk-1} implies that $$|u^{**}(0,0)-p(t^*)|\geq \delta>0,$$ which is a contradiction.

Next, we prove
\begin{equation}\label{V-neg-limit}
\lim_{x\to-\infty}V(t,x)=0 \quad {\rm uniformly} \,\ {\rm in}\,\ t\in\R.
\end{equation}

First, by Lemma \ref{new-lm1}, we have 
\begin{equation}\label{psi-definition-eq3}
V(t, x)=\frac{\mu}{2\sqrt{\lambda}}\int_{\R}e^{-\sqrt{\lambda}|z|}U(t, x-z)dz.
\end{equation}
For any $\varepsilon>0$, since
\begin{equation}\label{finite-integral}
\frac{\sqrt{\lambda}}{2} \int_{\R} e^{-\sqrt{\lambda}|z|}dz=1
\end{equation}
and 
\begin{equation}\label{U-bddness}
\sup_{t\in\R, x\in\R}|U(t, x)|\leq \frac{\bar r}{b-\chi\mu},
\end{equation}
there is $R>0$ large enough such that
\begin{equation}\label{V-eq-part1}
\frac{\mu}{2\sqrt{\lambda}}\int_{\R  \backslash[-R,R]} e^{-\sqrt{\lambda}|z|}|U(t, x-z)|dz<\frac{\varepsilon}{2}, \quad \forall\,\, t\in\R, \,\ x\in\R.
\end{equation}
According to \eqref{U-limit}, there is $y_0<0$, such that $U(t, y)<\frac{\lambda}{2\mu}\varepsilon$ for any $y\leq y_0$, any $t\in\R$. Note that 
$x-z\leq x+R\leq y_0$ for any $z\in [-R, R]$ and $x\leq y_0-R$. Hence,
$$
\sup_{t\in\R, z\in[-R, R]} U(t, x-z)<\frac{\lambda}{2\mu}\varepsilon, \quad \forall\,\ x\leq y_0-R.
$$
This implies that
\begin{equation}\label{V-eq-part2}
\frac{\mu}{2\sqrt{\lambda}}\int_{-R}^{R}e^{-\sqrt{\lambda}|z|}U(t, x-z)dz<\frac{\varepsilon}{2} \quad \forall\,\ t\in\R,\,\ x\leq y_0-R.
\end{equation}
Thus, \eqref{V-neg-limit} follows from \eqref{psi-definition-eq3}, \eqref{V-eq-part1} and \eqref{V-eq-part2}.

Finally, we prove
\begin{equation}\label{V-posi-limit}
\lim_{x\to\infty}|V(t,x)-\frac{\mu}{\lambda} p(t)|=0 \quad {\rm uniformly} \,\ {\rm in}\,\ t\in\R.
\end{equation}
In view of  \eqref{finite-integral}, \eqref{U-bddness} and the boundedness of $p(t)$, it follows that  for any $\varepsilon>0$,
there is $R>0$ large enough such that
\begin{equation}\label{V-eq-part1-1}
\frac{\mu}{2\sqrt{\lambda}}\int_{\R  \backslash[-R,R]} e^{-\sqrt{\lambda}|z|}|U(t, x-z)-p(t)|dz<\frac{\varepsilon}{2}, \quad \forall\,\, t\in\R, \,\ x\in\R.
\end{equation}
Observe that
\begin{align}\label{V-eq-part3}
|V(t,x)-\frac{\mu}{\lambda} p(t)|
&\leq \frac{\mu}{2\sqrt{\lambda}}\int_{\R}e^{-\sqrt{\lambda}|z|}|U(t, x-z)-p(t)|dz\cr
&\leq \frac{\mu}{2\sqrt{\lambda}}\int_{-R}^{R}e^{-\sqrt{\lambda}|z|}|U(t, x-z)-p(t)|dz\cr
&+\frac{\mu}{2\sqrt{\lambda}}\int_{\R  \backslash[-R,R]} e^{-\sqrt{\lambda}|z|}|U(t, x-z)-p(t)|dz.
\end{align}
From \eqref{kk-0}, there is $x_0>0$ such that
$$
|U(t, x)-p(t)|\leq \frac{\lambda}{2\mu}\varepsilon \quad \forall\,\, x\geq x_0,\,\ t\in\R.
$$
Note that 
$x-z\geq x-R\geq x_0$ for any $z\in [-R, R]$ and $x\geq x_0+R$. Hence,
\begin{equation}\label{sup-U}
\sup_{t\in\R, z\in[-R, R]} |U(t, x-z)-p(t)|\leq \frac{\lambda}{2\mu}\varepsilon, \quad \forall\,\ x\geq x_0+R.
\end{equation}
Combining \eqref{finite-integral}, \eqref{V-eq-part1-1}, \eqref{V-eq-part3} and \eqref{sup-U}, we then obtain 
$$
\sup_{t\in\R}|V(t,x)-\frac{\mu}{\lambda} p(t)|\leq \varepsilon
$$
whenever $x\geq x_0+R$. This completes the proof of \eqref{V-posi-limit}.

\end{proof}

\begin{proof}[Proof of Theorem \ref{nonexistence-forced-wave}]
Assume, by contradiction, that \eqref{Keller-Segel-eq0} has a time $\omega$-periodic forced wave solution $(u(t,x),v(t,x))=(U(t, x-ct),V(t, x-ct))$
connecting $(0, 0)$ and $(p(t),\frac{\mu}{\lambda}p(t))$ with $c<-c^{*}$. Note that
$$
0<U(0, x)\leq U^{+}(x), \quad \forall\,\, x\in\R,
$$
where $U^{+}(x)$ is given in \eqref{U-Def}. Let $\kappa$ be given in Lemma \ref{lem-001-1}. Since $c<-c^{*}$, we have $\kappa<\theta$. Therefore, we can choose $M\gg \frac{\bar r}{b-\chi\mu}$ such that
$$
U^{+}(x)<\min\{M, Me^{\kappa x}\} \quad \forall\,\, x\in\R.
$$
It then follows that  
$$
0<U(0, x)<\min\{M, Me^{\kappa x}\} \quad \forall\,\, x\in\R.
$$
By the arguments similar to those for \eqref{u-left}, it follows  that
 for any $\varepsilon>0$,
\begin{equation*}
\lim_{t\to\infty}\sup_{x\le (-c^*-\varepsilon)t} u(t,x)=0,
\end{equation*}
and hence, 
$$
0<U(t, (-c^*-\varepsilon-c)t)\leq \sup_{x\le (-c^*-\varepsilon)t}U(t, x-ct)=\sup_{x\le (-c^*-\varepsilon)t} u(t,x), \quad \forall\,\, t>0.
$$
Letting $t\to\infty$, we have 
$$
\lim_{t\to\infty}U(t, (-c^*-\varepsilon-c)t)=0
$$
Since $\lim_{x\to\infty}|U(t, x)-p(t)|=0$ uniformly in $t\in\R$, we must have $c>-c^{*}-\varepsilon$. Letting $\varepsilon\to 0$, we obtain  $c\geq -c^{*}$, a contradiction. 
\end{proof}

%\section*{Statements and Declarations}
%\begin{itemize}
%\item {{\bf Data Availability}:} Data sharing not applicable to this article as no datasets were generated or analysed during the current study.
%\item {\bf Competing Interests:} Authors are required to disclose financial or non-financial interests that are directly or indirectly related to the work submitted for publication.
%\end{itemize}

%\noindent{\bf Data Availability:} Data sharing is not applicable to this article as no datasets were generated or analyzed during the current study.

%\medskip

%\noindent{\bf Competing Interests:} The authors have no relevant financial or non-financial interests to disclose.

%\medskip

%\noindent{\bf Author Contribution:} Shuwen Xue wrote the main manuscript text. Xiao-Qiang Zhao revised the manuscript. All authors read and approved the final manuscript.

%\medskip

%\noindent{\bf Funding Declaration:} Xiao-Qiang Zhao's research is partially supported by the NSERC of Canada (RGPIN-2019-05648 and RGPIN-2025-04963). 

\end{document}